\documentclass[a4paper,10pt,twoside]{article}

\usepackage[utf8]{inputenc}
\usepackage{amsmath, amssymb, amsthm}
\usepackage{amstext, amsfonts, a4}
\usepackage{dsfont}
\usepackage{latexsym}
\usepackage{mathrsfs}
\usepackage[all,cmtip]{xy}
\usepackage{graphicx}
\usepackage{enumerate}
\usepackage{hyperref,color}
\usepackage{stmaryrd}
\usepackage[shortlabels]{enumitem}
\usepackage{bbm}

\usepackage{url}
\usepackage[pdftex,a4paper,left=3cm,right=3cm,top=3cm,bottom=3cm]{geometry}

\def\ra{\rightarrow}

\def\lra{\longrightarrow}

\def\hra{\hookrightarrow}
\def\lmapsto{\longmapsto}

 \def\sB{\mathscr{B}}  
  
  \def\sL{\mathscr{L}}

\def\bbA{\mathbb{A}}
\def\bbF{\mathbb{F}}\def\bbG{\mathbb{G}}

\def\bbN{\mathbb{N}}\def\bbP{\mathbb{P}}
\def\bbQ{\mathbb{Q}}\def\bbR{\mathbb{R}}\def\bbT{\mathbb{T}}

\def\bbZ{\mathbb{Z}}

\def\cA{\mathcal{A}}\def\cC{\mathcal{C}}
\def\cG{\mathcal{G}}\def\cH{\mathcal{H}}
\def\cL{\mathcal{L}}
\def\cO{\mathcal{O}}\def\cP{\mathcal{P}}
\def\cS{\mathcal{S}}\def\cT{\mathcal{T}}

\def\bfG{\mathbf{G}}
\def\bfL{\mathbf{L}}

\def\bfT{\mathbf{T}}

\def\dsS{\mathds{S}}\def\dsU{\mathds{U}}\def\dsZ{\mathds{Z}}

\def\bfc{\mathbf{c}}
\def\bfd{\mathbf{d}}

\def\bfq{\mathbf{q}}

\def\fs{\mathfrak{s}}

\def\fC{\mathfrak{C}}

\def\whB{\widehat{B}}
\def\whG{\widehat{G}}
\def\whI{\widehat{I}}\def\whL{\widehat{L}}

\def\whT{\widehat{T}}

\def\hG{\mathbf{\widehat{G}}}
\def\hT{\mathbf{\widehat{T}}}
\def\hB{\mathbf{\widehat{B}}}
\def\hL{\mathbf{\widehat{L}}}

\DeclareMathOperator{\ab}{ab}
\DeclareMathOperator{\ad}{ad}
\DeclareMathOperator{\Ad}{Ad}
\DeclareMathOperator{\aff}{aff}
\DeclareMathOperator{\Aut}{Aut}

\DeclareMathOperator{\As}{As}

\DeclareMathOperator{\cox}{cox}

\DeclareMathOperator{\der}{der}
\DeclareMathOperator{\diag}{diag}

\DeclareMathOperator{\ev}{ev}

\DeclareMathOperator{\fin}{f}

\DeclareMathOperator{\gr}{gr}

\DeclareMathOperator{\Hom}{Hom}

\DeclareMathOperator{\id}{id}

\DeclareMathOperator{\ind}{ind}

\DeclareMathOperator{\pos}{pos}
\DeclareMathOperator{\pr}{pr}

\DeclareMathOperator{\rank}{rank}

\DeclareMathOperator{\reg}{reg}

\DeclareMathOperator{\scc}{sc}

\DeclareMathOperator{\Spec}{Spec}

\DeclareMathOperator{\Sym}{Sym}
\DeclareMathOperator{\sym}{sym}

\DeclareMathOperator{\unr}{unr}

\def\lan{\langle}

\def\ran{\rangle}

\newtheorem{counter}[subsection]{$\!\!$}
\newenvironment{Def}{\begin{counter} {\bf Definition.}}{\end{counter}}

\newenvironment{Prop}{\begin{counter} {\bf Proposition.}}{\end{counter}}

\newenvironment{Pt}{\begin{counter} \rm}{\end{counter}}

\newenvironment{Proof}{{\flushleft \bf Proof :}}{\hfill $\square$ \vspace{5mm}}

\newtheorem{counter*}[subsubsection]{$\!\!$}
\newenvironment{Def*}{\begin{counter*} {\bf Definition.}}{\end{counter*}}
\newenvironment{Not*}{\begin{counter*} \rm {\bf Notation.}}{\end{counter*}}
\newenvironment{Notss*}{\begin{counter*} \rm {\bf Notations.}}{\end{counter*}}
\newenvironment{DefNot*}{\begin{counter*} \rm {\bf Definition-Notation.}}{\end{counter*}}
\newenvironment{Nots*}{\begin{counter*} \rm {\bf Notations.}}{\end{counter*}}
\newenvironment{Prop*}{\begin{counter*} {\bf Proposition.}}{\end{counter*}}
\newenvironment{Lem*}{\begin{counter*} {\bf Lemma.}}{\end{counter*}}
\newenvironment{Cor*}{\begin{counter*} {\bf Corollary.}}{\end{counter*}}
\newenvironment{Th*}{\begin{counter*} {\bf Theorem.}}{\end{counter*}}
\newenvironment{Rem*}{\begin{counter*} \rm {\bf Remark.}}{\end{counter*}}
\newenvironment{Ex*}{\begin{counter*} \rm {\bf Example.}}{\end{counter*}}
\newenvironment{Exs*}{\begin{counter*} \rm {\bf Examples.}}{\end{counter*}}
\newenvironment{Pt*}{\begin{counter*} \rm}{\end{counter*}}
\newenvironment{Q*}{\begin{counter*} \rm {\bf Question.}}{\end{counter*}}

\newcommand{\iso}{\stackrel{\sim}{\longrightarrow}}

\title{\huge{\textbf{pro-$p$ Iwahori Hecke algebras \\ and the dual Vinberg monoid}}}
\author{Tobias SCHMIDT}
\date{\today}

\begin{document}

\maketitle

\begin{abstract}
Let $\bfG$ be a split reductive group over the integers, $F$ a $p$-adic local field with residue field $\bbF_q$ and $G=\bfG(F)$. We relate the pro-$p$-Iwahori Hecke algebra $H$ of $G$ over $\bbF_q$ to the Vinberg monoid of the dual group $\widehat{\bfG}$ and study this relation. As an application, in the $GL_n$-case and for $F/ \bbQ_p$ unramified, we derive a parametrization of $\Spec Z$ by semisimple $n$-dimensional representations of the absolute Galois group ${\rm Gal}(\overline{F}/F)$, generalizing the case $n=2$ of \cite{PS25}. Here $Z$ denotes the center of $H$. 
\end{abstract}

\tableofcontents

\section{Introduction}

Let $\bfG$ be a split connected reductive group over $\bbZ$, with split maximal torus $\bfT$. Let $F/\bbQ_p$ be a finite extension with residue field $\bbF_q$ and $G=\bfG(F)$. 
The pro-$p$-Iwahori Hecke algebra $\cH^{(1)}_{\bbF_q}$ of the reductive $p$-adic group $G$ with coefficients in $\bbF_q$, as introduced and studied by Vignéras \cite{V16} is the convolution algebra over $\bbF_q$ on the set of 
double cosets of $G$ relative to the choice of a pro-$p$-Iwahori subgroup in $G$. The algebra $\cH^{(1)}_{\bbF_q}$ (and derived versions thereof) is expected to have strong applications to the categorical mod $p$ Langlands program, e.g. \cite{H16, EGH22}. Let 
$Z(\cH^{(1)}_{\bbF_q})$ be the center of $\cH^{(1)}_{\bbF_q}$. Let $\zeta$ be a central character of $G$ and $\omega$ the mod $p$ cyclotomic character of the absolute Galois group ${\rm Gal}(\overline{F}/F)$ of the $p$-adic field $F$. When $\bfG=GL_2$, it is shown in \cite{PS25} that the $\zeta$-part of the center admits a quotient morphism of $\bbF_q$-schemes
$$ \sL_\zeta: \Spec Z(\cH^{(1)}_{\bbF_q})_\zeta \lra X(q)_{\bbF_q,\zeta}$$ 
onto a certain projective algebraic curve $X(q)_{\bbF_q,\zeta}$, which parametrizes the semisimple $2$-dimensional mod $p$ representations of 
${\rm Gal}(\overline{F}/F)$ with determinant $\omega\zeta$.  Combined with the theory of the spherical representation \cite{PS23}, the morphism $\sL_\zeta$ recovers Grosse-Klönne's bijection in dimension $2$ between supersingular mod $p$ Hecke modules and irreducible Galois representations \cite{GK18}. In the case $F=\bbQ_p$, it recovers Breuil's semisimple mod $p$ local Langlands correspondence \cite{Br03} for the group $GL_2(\bbQ_p)$. It is natural to ask how this situation generalizes from $GL_2$ to more general groups $\bfG$. 

\vskip5pt
In the present paper, we generalize a suitable stratified version of the morphism $\sL_\zeta$ to 
the situation where  $\bfG=GL_n$ and $F/\bbQ_p$ is unramified. To explain how this works,
let us go back to the case $n=2$. In this case, both 
$\Spec Z(\cH^{(1)}_{\bbF_q})_\zeta$ and the curve 
$X(q)_{\bbF_q,\zeta}$ admit canonical stratifications by locally closed subvarieties. 
In fact, the stratifications consist only of two strata, given on $X(q)_{\bbF_q,\zeta}$ by the closed subspace of irreducible $2$-dimensional Galois representations and its open complement. 
The morphism $\sL_\zeta$ preserves the stratifications and is given on the open stratum by quotienting out a suitable dot action of the Weyl group. Now assume $\bfG=GL_n.$
We construct a canonical stratification of 
$\Xi:=\Spec Z(\cH^{(1)}_{\bbF_q})$
by locally closed subvarieties  $\Xi_{\bfL}$, indexed by the standard Levi subgroups $\bfL\subset \bfG$. Each stratum 
$\Xi_{\bfL}$ admits a suitable dot action of the Weyl group $W(\bfL):=N_{\bfG}(\bfL)/\bfL$. There is a $W(\bfL)$-stable subset $\Xi_{\bfL,\rm gen}$ consisting of "generic" (in a precise sense) connected components of $\Xi_{\bfL}$. The proportion of generic components tends to $1$, as $p$ tends to infinity. On the other hand, the semisimple $n$-dimensional mod $p$ representations of the absolute Galois group ${\rm Gal}(\overline{F}/F)$ are stratified along the standard Levi subgroups $\widehat{\bfL}$ in the dual group $\widehat{\bfG}$ (the duality $\bfG\leftrightarrow\widehat{\bfG}$ interchanges $\bfL$ and $\widehat{\bfL}$): there is an affine $\bbF_q$-scheme $X_{[\hL]}$, whose regular components (in a precise sense) parametrize the semisimple $n$-dimensional mod $p$ representations of ${\rm Gal}(\overline{F}/F)$ "with values in $\hL$". Here is the main result of the paper (Thm. \ref{cor_main_appli} and section \ref{section_GL2}).

\vskip5pt
{\bf Theorem.} {\it Let $\bfG=GL_n$ and $F/\bbQ_p$ unramified.\vskip5pt (i) For each Levi subgroup $\bfL\subset \bfG$ there is a morphism of affine $\bbF_q$-schemes $$\Xi_{\bfL,\rm gen}\ra X_{[\hL]},$$
which is constant on $W(\bfL)$-orbits. 

\vskip5pt 
(ii) Let $\bfG=GL_2$. The morphism of (i) extends to $\Xi_{\bfL}$ and induces an isomorphism  $$\Xi_{\bfL}/W(\bfL)\simeq X_{[\hL]}$$ for each Levi subgroup $\bfL\subset \bfG$.
Precomposition with the projection $\Xi_{\bfL}\ra \Xi_{\bfL}/W(\bfL)$ and restriction to the $\zeta$-part recovers the morphism $\sL_\zeta \mid_{ \Xi_{\bfL}}$.}
\vskip5pt

In the following, we explain how the morphism in part (i) is constructed. As in \cite{PS25}, the main tool for passing from the Hecke algebra of $G$ to Galois representations into the dual group $\widehat{\bfG}$ is the dual Vinberg monoid, or rather a suitable toral subvariety of what one may call the dual Vinberg-Zhu monoid \cite{Z20}. To give more details, let $V_{\mathbf{\whG}}$ be Vinberg's monoid \cite{Vin95} associated with the dual reductive group $\widehat{\bfG}$. Let $\rho_{\rm ad}$ denote the sum over the fundamental coweights of the adjoint group of $\widehat{\bfG}$, as well as its unique extension to $\bbA^1$. The Vinberg-Zhu monoid is the flat monoid $V_{\mathbf{\whG},\rho_{\rm ad}}\ra\bbA^1$ equal to the pull-back of  $V_{\mathbf{\whG}}$ along $\rho_{\rm ad}$. Its $1$-fiber recovers the group $\mathbf{\whG}$. Let $\widehat{\bfT}\subset \widehat{\bfG}$ denote the dual torus. It gives rise to a "toral" subvariety $V_{\mathbf{\whT}\subset\mathbf{\whG},\rho_{\ad}}\ra\bbA^1$ of $V_{\mathbf{\whG}}$. It has a natural $W$-action. Its $1$-fiber recovers $\widehat{\bfT}$ and its $0$-fiber 
$V_{\hT\subset\hG,0}$ is a certain semigroup stable under the action of $\widehat{\bfT}$.
The $\bbA^1$-monoid $V_{\mathbf{\whT}\subset\mathbf{\whG},\rho_{\ad}}$ is the correct version of the dual torus, when it comes to $\bfq$-versions of the classical (i.e. away from $\bfq=0$) Bernstein-Lusztig isomorphism \cite{L83, L89}. To make this more precise, write $\bbA^1=\Spec \bbZ[\bfq]$ with some indeterminate $\bfq$ and let 
$\cH^{(1)}(\bfq)$ be Vignéras {\it generic} 
pro-$p$-Iwahori Hecke algebra \cite{V16}, a suitable $\bbZ[\bfq]$-algebra defined by generators 
and relations,
specialising via $\bfq\mapsto q=0\in \bbF_q$ to the above $\bbF_q$-algebra $\cH^{(1)}_{\bbF_q}$. The algebra $\cH^{(1)}(\bfq)$ comes with a maximal commutative subring 
$\cA^{(1)}(\bfq)\subset\cH^{(1)}(\bfq)$ (depending on a choice of orientation), which itself has a canonical $W$-action. Similarly, there is a version $\cA(\bfq)\subset\cH(\bfq)$ of this situation on the level of the usual Iwahori Hecke algebra.
One of our main secondary results (Thm. \ref{ThBernsteingeneric}) says that 
the identity $X_*(\mathbf{T})=X^*(\mathbf{\whT})$ extends to a canonical $W$-equivariant isomorphism of 
$\bbZ[\bfq]$-algebras 
$$
\sB(\bfq):\cA(\bfq) \iso \bbZ[V_{\mathbf{\whT}\subset\mathbf{\whG},\rho_{\ad}}]$$
with the ring of functions on $V_{\mathbf{\whT}\subset\mathbf{\whG},\rho_{\ad}}$.
This $\bfq$-Bernstein-Lusztig isomorphism for general $\bfG$ generalizes the $GL_2$-case appearing in \cite{PS23}.
Passing to the $0$-fibre, we obtain the following 
semigroup over $\bbF_q$
$$V^{(1)}_{\hT\subset\hG,0,\bbF_q}:=
\hT(\bbF_q)\times_{\bbZ} V_{\hT\subset\hG,0},
$$ 
"augmented" by the finite torus $\hT(\bbF_q)$. It has its diagonal
$W$-action.
The above morphism $\sB(\bfq)$ induces 
a $W$-equivariant isomorphism of $\bbF_q$-algebras between 
$
\cA^{(1)}(0)_{\bbF_q}$ and the ring of functions on $V^{(1)}_{\hT\subset\hG,0,\bbF_q}$.
In this situation, Vignéras work implies that the $W$-invariants $\cA^{(1)}(0)_{\bbF_q} ^W$ equal the center $Z(\cH^{(1)}_{\bbF_q})$, whence an isomorphism 
$$ \Spec  Z(\cH^{(1)}_{\bbF_q}) \iso V^{(1)}_{\hT\subset\hG,0,\bbF_q}/W.$$ 

It is this isomorphism, which allows us to pass from the Hecke algebra of $G$ to the side of the dual group $\widehat{\bfG}$ and Galois representations with values in that group. 
The fact that $V^{(1)}_{\hT\subset\hG,0,\bbF_q}$ is a replacement of the dual torus $\widehat{\bfT}$ rather than the full dual group $\widehat{\bfG}$ explains, as we like to think, why the morphisms in our main theorem above take values in {\it semisimple} Galois representations, as opposed to more general kinds of Galois representations. 
The scheme $V_{\hT\subset\hG,0,\bbF_q}/W$ is a connected component stable under $\widehat{\bfT}$ of the scheme\footnote{The "Satake scheme" $V^{(1)}_{\hT\subset\hG,0,\bbF_q}/W$ was denoted $S(q)$ in \cite{PS25}.} 
$S:=V^{(1)}_{\hT\subset\hG,0,\bbF_q}/W$ and its 
$\widehat{\bfT}$-orbits are naturally indexed by the dual Levi subgroups $\widehat{\bfL}\subset \widehat{\bfG}$.
This gives a stratification of $S$ by locally closed subvarieties 
$S_{\widehat{\bfL}}\subset S.$
The inverse image of $S_{\widehat{\bfL}}$ in $\Spec  Z(\cH^{(1)}_{\bbF_q})$ is by definition our $\Xi_{\bfL}$. We show that the scheme $S_{\widehat{\bfL}}$ is a finite disjoint union of copies of the torus 
$\widehat{\bfL}^{\rm ab}$ indexed by the orbits in 
$\widehat{\bfT}(\bbF_q)/W_{\hL} $ relative to the corresponding parabolic subgroup $W_{\hL}$ of $W$. Jantzen's parametrization of dual Deligne-Lustzig pairs \cite{J81} restricted to Coxeter classes and adapted to weights of finite tori yields a map 
from the subset of generic orbits in $\widehat{\bfT}(\bbF_q)/W_{\hL} $ (in a suitable precise sense) to Frobenius stable semisimple conjugacy classes in $\widehat{\bfG}$, i.e. to 
mod $p$ tame inertial types of ${\rm Gal}(\overline{F}/F)$ (it is here, where we assume that $F/\bbQ_p$ is unramified). Adding the torus $\widehat{\bfL}^{\rm ab}$ corresponds to adding  "the determinant of Frobenius" and produces a map 
$S_{\widehat{\bfL},\rm gen}\ra X_{[\hL]}$ on a subset $S_{\widehat{\bfL},\rm gen}$ of generic components of 
$S_{\widehat{\bfL}}$. Pulling back this map to $\Xi_{\bfL}$ produces the map in the main theorem above.  For more details, we refer to the main body of the paper.

\vskip5pt 

In the situation of the theorem, any coherent sheaf $\cA$ on $\Xi$ with an $\cH^{(1)}_{\bbF_q}$-action along fibers, attaches at each point $z\in\Xi_{\bfL}$ a set of finitely many irreducible $\cH^{(1)}_{\bbF_q}$-modules (the Jordan-Hölder factors of the fiber $\cA_z$) to the corresponding Galois representation in $X_{[\hL]}.$
In the $GL_2$-case, taking for $\cA$ the spherical representation \cite{PS23} realizes the compatibilities with \cite{GK18} and \cite{Br03} alluded to in the beginning of this introduction. For the moment, a theory of the spherical representation for $GL_n$ is not available, we hope to come back to this in future work. 

\vskip5pt 

In the remaining part of the introduction, we briefly describe the content of the individual sections. 
In section $2$, we recall the construction of the Vinberg monoid $V_{\mathbf{\whG}}$, its restriction to the diagonal $V_{\mathbf{\whG},\rho_{\rm ad}}$ due to Zhu and the toral subvariety $V_{\hT\subset\hG}.$ We establish some geometric properties of its special fibre $V_{\hT\subset\hG,0}$. In section 3 we discuss some properties of mod $p$ tame representations of ${\rm Gal}(\overline{F}/F)$ and mod $p$ tame inertial types with values in $\mathbf{\whG}.$ This includes, in the case of $GL_n$, a classification of the irreducible representations in terms of Frobenius and monodromy. In subsection $4$ we discuss Deligne-Lusztig pairs and explain how Jantzen's parametrization of such pairs, evaluated on Coxeter classes, gives rise to a map from algebraic weights to mod $p$ tame inertial types. We discuss dot actions on finite tori in the $GL_n$-case, which allows us to apply Jantzen's map to weights of finite tori. This a priori depends on choices, which we explain in 4.3. Finally, we establish in 4.4 the map 
 $S_{\hL,\rm gen}\ra X_{[\hL]},$ by reducing it to suitable map between connected components and applying Jantzen's map. In the last subsection 4.5 we discuss the special case $GL_2$. In the final section $5$ we establish the above $\bfq$-Bernstein-Lusztig isomorphism  $\sB(\bfq)$ and deduce the application to pro-$p$-Iwahori Hecke algebras and their centers. We assemble some material on tori, roots and weights in an appendix. 

\vskip5pt 

Acknowledgements: The author thanks Elmar Grosse-Klönne, Cédric Pépin, Peter Schneider and Marie-France Vignéras for interesting discussions on this work over the years. He thanks in particular Cédric Pépin for numerous exchanges, many ideas presented here emerged in several common discussions.


\section{The dual Vinberg monoid}

We recall the algebraic construction of Vinberg's monoid via the canonical filtration and several related constructions, following Zhu's exposition in \cite{Z20}. This also allows us to set up the notation. Let $\mathbf{\whG}$ be a split connected reductive group scheme over $\bbZ$, with maximal torus $ \mathbf{\whT}$ contained in a Borel subgroup scheme $\mathbf{\whT} \subset \mathbf{\whB}$. Let $W$ be the Weyl group of the pair $(\mathbf{\whG},\mathbf{\whT})$.

\subsection{Definition}

\begin{Pt*}
Let $Z_{\mathbf{\whG}}$ be the center of $\mathbf{\whG}$ and $\mathbf{\whT}_{\ad}:=\mathbf{\whT}/Z_{\mathbf{\whG}}$. Let $Z_{\mathbf{\whG}}\subset \mathbf{\whG}\times \mathbf{\whT}$ be the diagonal inclusion. We then have a  $Z_{\mathbf{\whG}}$-action on $\mathbf{\whG}\times \mathbf{\whT}$ via $z.(g,t):=(zg,zt)$. Let $\mathbf{\whG}\times^{Z_{\mathbf{\whG}}}\mathbf{\whT}:=(\mathbf{\whG}\times \mathbf{\whT})/Z_{\mathbf{\whG}}$. It is an affine $\bbZ$-group scheme equipped with a faithfully flat group homomorphism 
$$
\xymatrix{
\mathbf{\whG}\times^{Z_{\mathbf{\whG}}}\mathbf{\whT}\ar[d]^{\pr_2} \\
\mathbf{\whT}_{\ad}.
}
$$

The group $X^*(\mathbf{\whT}_{\ad})$ is freely generated by the set of simple roots of $(\mathbf{\whG},\mathbf{\whT})$, which in turn define a set of invertible coordinates on 
the scheme $\mathbf{\whT}_{\ad}$. The $\bbZ$-point $1:=(1,\dots,1)$ is the unit of the group $\mathbf{\whT}_{\ad}$. The fiber $\pr_2^{-1}(1)$ identifies with  $\mathbf{\whG}$ embedded into $\mathbf{\whG}\times^{Z_{\mathbf{\whG}}}\mathbf{\whT}$ as $\mathbf{\whG}\times\{1_{\mathbf{\whT}}\}$.
\end{Pt*}

\begin{Pt*}
The canonical projection 
$\mathbf{\whT}\ra\mathbf{\whT}_{\ad}$ induces an injection $X^*(\mathbf{\whT}_{\ad})\ra X^*(\mathbf{\whT})$ identifying $X^*(\mathbf{\whT}_{\ad})$ with the subgroup of $X^*(\mathbf{\whT})$ generated by the set of roots of $(\mathbf{\whG},\mathbf{\whT})$. 
 If a character $e^{\nu}\in X^*(\mathbf{\whT})$ factors through the quotient $\mathbf{\whT}_{\ad}$ of $\mathbf{\whT}$, we write $\bar{e}^{\nu}$ for the corresponding character of $\mathbf{\whT}_{\ad}$. Let $X^*(\mathbf{\whT}_{\ad})_{\pos}\subset X^*(\mathbf{\whT}_{\ad})$ be the submonoid generated by the subset of positive roots with respect to $\mathbf{\whB}$. Then
$$
\mathbf{\whT}_{\ad}=\Spec(\bbZ[X^*(\mathbf{\whT}_{\ad})])\lra \mathbf{\whT}_{\ad}^+:= \Spec(\bbZ[X^*(\mathbf{\whT}_{\ad})_{\pos}])
$$
is an open immersion, and $\mathbf{\whT}_{\ad}^+$ is a diagonalizable affine $\bbZ$-monoid scheme with group of invertible elements the torus 
$\mathbf{\whT}_{\ad}$. The monoid $X^*(\mathbf{\whT}_{\ad})_{\pos}$ is freely generated by the set of simple roots, which define a set of coordinates on
$\mathbf{\whT}_{\ad}^+$ extending the one on $\mathbf{\whT}_{\ad}$. In particular $1=(1,\dots,1)$ is the unit of the monoid $\mathbf{\whT}_{\ad}^+$. The $\bbZ$-point $0:=(0,\dots,0)$ is an absorbing element. 
There is the usual partial order on $X^*(\mathbf{\whT})$ given by 
$\lambda_1\leq \lambda_2$ if $\lambda_1-\lambda_2\in X^*(\mathbf{\whT}_{\ad})_{\pos}.$
We let $X^*(\mathbf{\whT}_{\ad})^+_{\pos}$ be the submonoid of $X^*(\mathbf{\whT})$ generated by $X^*(\mathbf{\whT}_{\ad})_{\pos}$ and the dominant (relative to $\mathbf{\whB}$) weights $X^*(\mathbf{\whT})^+$.
\end{Pt*}

\begin{Pt*}\label{DefVinberg}\label{VGcartsquares}
The ring of functions $\bbZ[\mathbf{\whG}]$ is a $\mathbf{\whG}\times\mathbf{\whG}$-module via left and right translation of functions. It admits a filtration ${\rm fil}_\bullet\bbZ[\mathbf{\whG}]$ by $\mathbf{\whG}\times\mathbf{\whG}$-submodules indexed by $X^*(\mathbf{\whT}_{\ad})^+_{\pos}$, where 
$${\rm fil}_{\lambda}\bbZ[\mathbf{\whG}]$$
is the largest $\mathbf{\whG}\times\mathbf{\whG}$-submodule $V$ of $\bbZ[\mathbf{\whG}]$ such that $V(\nu)\neq 0$ implies $\nu \leq (\lambda^*,\lambda)$ for any weight $\nu$ of 
$\hT\times\hT$. Here $\lambda^*=-w_o(\lambda)$, where $w_o$ is the longest element of $W$ and $\leq$ is applied to each of the two components. We let 
$$R_{X^*(\mathbf{\whT}_{\ad})^+_{\pos}} \bbZ[\mathbf{\whG}]:=
\oplus_{\lambda} {\rm fil}_{\lambda}\bbZ[\mathbf{\whG}]e^{\lambda}$$
be the usual Rees ring associated to the above filtration (with formal symbols $e^\lambda$).
The natural inclusion $X^*(\mathbf{\whT}_{\ad})_{\pos}\subset X^*(\mathbf{\whT}_{\ad})^+_{\pos}$ induces a ring homomorphism
$$\bbZ[X^*(\mathbf{\whT}_{\ad})_{\pos}]\longrightarrow R_{X^*(\mathbf{\whT}_{\ad})^+_{\pos}} \bbZ[\mathbf{\whG}], e^{\lambda}\mapsto e^{\lambda},$$
which is faithfully flat. Both the source and the target are coalgebras and the homomorphism is in fact a homomorphism of coalgebras. 
The spectrum 
$$V_{\mathbf{\whG}}:=\Spec R_{X^*(\mathbf{\whT}_{\ad})^+_{\pos}} \bbZ[\mathbf{\whG}].$$
becomes thus an affine $\bbZ$-monoid scheme equipped with a faithfully flat monoid homomorphism 
$$
\xymatrix{
V_{\mathbf{\whG}} \ar[d]^d \\
\mathbf{\whT}_{\ad}^+
}
$$
One may verify that the inverse image
$d^{-1}(\mathbf{\whT}_{\ad})$ coincides
with the group of invertible elements of $V_{\mathbf{\whG}}$ and that the latter group is isomorphic, as a group scheme over $\mathbf{\whT}_{\ad}$, to $\mathbf{\whG}\times^{Z_{\mathbf{\whG}}}\mathbf{\whT}$ (with second projection). The monoid $V_{\mathbf{\whG}}$ is called the \emph{Vinberg monoid} of $\mathbf{\whG}$.

By general properties of the Rees ring, the special fibre $d^{-1}(0)$ identifies with the so-called {\it asymptotic cone} $\As_{\mathbf{\whG}}$ of $\mathbf{\whG}$ given by 
$\Spec \gr_\bullet \bbZ[\mathbf{\whG}]$, where $\gr_\bullet$ refers to the associated graded ring of a filtered ring. In particular, there are cartesian squares
$$
\xymatrix{
\mathbf{\whG} \ar[r] \ar[d]  &\ar[r] \ar[d]\mathbf{\whG}\times^{Z_{\mathbf{\whG}}}\mathbf{\whT} \ar[r] & V_{\mathbf{\whG}}\ar[r] \ar[d]^d & \As_{\mathbf{\whG}} \ar[d] \\
\Spec(\bbZ) \ar[r]^1 & \mathbf{\whT}_{\ad} \ar[r] &  \mathbf{\whT}_{\ad}^+ & \Spec(\bbZ) \ar[l]_0,
}
$$
the two external ones being closed immersions and the middle one an open immersion.
\end{Pt*}

\begin{Pt*}\label{rhoad+}
Let $\rho_{\ad}:\bbG_m\ra \mathbf{\whT}_{\ad}$ be the sum of the fundamental coweights of the adjoint group of $\mathbf{\whG}$. It is a cocharacter of the torus $\mathbf{\whT}_{\ad}$, which is dominant. Consequently, it extends uniquely to a monoid homomorphism
$$
\xymatrix{
\rho_{\ad}^+:\bbA^1\ar[r] & \mathbf{\whT}_{\ad}^+.
}
$$
Choosing a variable $\bfq$ on $\bbA^1$, it is given by the ring homomorphism 
$$\bbZ[\mathbf{\whT}_{\ad}^+]\lra \bbZ[\bfq], \bar{e}^\mu\mapsto \bfq^{\langle\mu,\rho_{\ad}\rangle}$$  for $\mu\in X^*(\mathbf{\whT}_{\ad}^+)=X^*(\mathbf{\whT}_{\ad})_{\pos}$.

\end{Pt*}

\begin{Def*}\label{Zmon}
We call the fiber product $V_{\mathbf{\whG},\rho_{\ad}}:=V_{\mathbf{\whG}}\times_{d, \mathbf{\whT}_{\ad}^+,\rho_{\ad}}\bbA^1$ the \emph{Vinberg-Zhu monoid}. It is equipped with the monoid homomorphism $d\rho_{\ad}:V_{\mathbf{\whG},\rho_{\ad}}\ra \bbA^1$ equal to the projection to the second factor:
$$
\xymatrix{
V_{\mathbf{\whG},\rho_{\ad}}\ar[r] \ar[d]_{d\rho_{\ad}} & V_{\mathbf{\whG}} \ar[d]^{d} \\
\bbA^1 \ar[r]^{\rho_{\ad}^+} & \mathbf{\whT}_{\ad}^+.
}
$$
\end{Def*}


\subsection{The toral submonoid $V_{\mathbf{\whT}\subset\mathbf{\whG}}$}\label{toralVin}

The diagonal inclusion $Z_{\mathbf{\whG}}\subset \mathbf{\whG}\times \mathbf{\whT}$ factors through $\mathbf{\whT}\times \mathbf{\whT}$, and 
$$\mathbf{\whT}\times^{Z_{\mathbf{\whG}}} \mathbf{\whT}:=(\mathbf{\whT}\times \mathbf{\whT})/Z_{\mathbf{\whG}}$$ is a closed subgroup
of $\mathbf{\whG}\times^{Z_{\mathbf{\whG}}} \mathbf{\whT}$. 

\begin{Def*}
The affine $\bbZ$-monoid scheme $V_{\mathbf{\whT}\subset\mathbf{\whG}}$ is the closure of $\mathbf{\whT}\times^{Z_{\mathbf{\whG}}} \mathbf{\whT}$ in $V_{\mathbf{\whG}}$.
\end{Def*}

The monoid $V_{\mathbf{\whT}\subset\mathbf{\whG}}$ is a submonoid of $V_{\mathbf{\whG}}$, whose group of invertible elements is the open subgroup 
$$
\mathbf{\whT}\times^{Z_{\mathbf{\whG}}} \mathbf{\whT}=V_{\mathbf{\whT}\subset\mathbf{\whG}}\cap d^{-1}(\mathbf{\whT}_{\ad})\subset V_{\mathbf{\whT}\subset\mathbf{\whG}}.
$$
Since moreover $\mathbf{\whT}\times^{Z_{\mathbf{\whG}}} \mathbf{\whT}$ is a quotient of $\mathbf{\whT}\times \mathbf{\whT}$, we have the inclusions
$$
\bbZ[V_{\mathbf{\whT}\subset\mathbf{\whG}}]\subset \bbZ[\mathbf{\whT}\times^{Z_{\mathbf{\whG}}} \mathbf{\whT}]\subset \bbZ[\mathbf{\whT}\times\mathbf{\whT}]=\bbZ[\mathbf{\whT}]\otimes \bbZ[\mathbf{\whT}].
$$
Then $\bbZ[V_{\mathbf{\whT}\subset\mathbf{\whG}}]$ is the following subring, cf. \cite[\S 1.3]{Z20}:
$$
\bbZ[V_{\mathbf{\whT}\subset\mathbf{\whG}}]=\bigoplus_{\substack{(\nu_1,\nu_2)\in X^*(\mathbf{\whT})^2 \\ \nu_2+(\nu_1)_{-}\in X^*(\mathbf{\whT}_{\ad})_{\pos}}}\bbZ (e_1^{\nu_1}\otimes e_2^{\nu_2}),
$$
where $\nu_{-}\in X^*(\mathbf{\whT})$ denotes the unique $W$-conjugate of $\nu\in X^*(\mathbf{\whT})$ which is \emph{antidominant}. In particular, the set indexing the direct sum is a submonoid of $X^*(\mathbf{\whT})^2$, and the $\bbZ$-monoid scheme $V_{\mathbf{\whT}\subset\mathbf{\whG}}$ is diagonalizable with monoid of characters
$$
X^*(V_{\mathbf{\whT}\subset\mathbf{\whG}})=\{(\nu_1,\nu_2)\in X^*(\mathbf{\whT})^2\ |\ \nu_2+(\nu_1)_{-}\in X^*(\mathbf{\whT}_{\ad})_{\pos}\}.
$$

\begin{Pt*}\label{VTGrhoad}
The monoid $V_{\mathbf{\whT}\subset\mathbf{\whG}}$ is equipped with the homomorphism $d:V_{\mathbf{\whT}\subset\mathbf{\whG}}\ra \mathbf{\whT}_{\ad}^+$ equal to the restriction of $d:V_{\mathbf{\whG}}\ra \mathbf{\whT}_{\ad}^+$. It corresponds to the ring homomorphism
\begin{eqnarray*}
\bbZ[X_*(\mathbf{\whT}_{\ad})_{\pos}] & \lra &\bbZ[V_{\mathbf{\whT}\subset\mathbf{\whG}}] \\
\bar{e}^{\nu} & \lmapsto & 1\otimes e_2^{\nu}
\end{eqnarray*}
In particular, it is faithfully flat.
Forming the cartesian square
$$
\xymatrix{
V_{\mathbf{\whT}\subset\mathbf{\whG},\rho_{\ad}}\ar[r] \ar[d]_{d\rho_{\ad}} & V_{\mathbf{\whT}\subset\mathbf{\whG}} \ar[d]^{d} \\
\bbA^1 \ar[r]^{\rho_{\ad}^+} & \mathbf{\whT}_{\ad}^+,
}
$$
we get a diagonalizable submonoid $V_{\mathbf{\whT}\subset\mathbf{\whG},\rho_{\ad}}\subset V_{\mathbf{\whG},\rho_{\ad}}$. 

\vskip5pt

As a last piece of notation, we will denote in the sequel the restriction of a function $e_1^{\nu_1}\otimes e_2^{\nu_2}$ on 
$V_{\mathbf{\whT}\subset\mathbf{\whG}}$ to 
$V_{\mathbf{\whT}\subset\mathbf{\whG},\rho_{\ad}}$ by $(e_1^{\nu_1}\otimes e_2^{\nu_2})|_{\rho_{\ad}} $. 

\end{Pt*}


\subsection{The Vinberg section}

\begin{Pt*}
Let $\fs: \mathbf{\whT}_{\ad}^+\ra V_{\mathbf{\whT}\subset\mathbf{\whG}}$ be the morphism of affine schemes corresponding to the ring morphism 
\begin{eqnarray*}
\bbZ[V_{\mathbf{\whT}\subset\mathbf{\whG}}] & \lra &\bbZ[X^*(\mathbf{\whT}_{\ad})_{\pos}] \\
e_1^{\nu_1}\otimes e_2^{\nu_2} & \lmapsto & \bar{e}^{\nu_1+\nu_2}
\end{eqnarray*}
Note here that indeed $\nu_1+\nu_2 \in X^*(\mathbf{\whT}_{\ad})_{\pos}$ since $\nu_1+\nu_2\geq \nu_2+(\nu_1)_-$. The morphism $\fs$ is a homomorphism of monoid schemes, which is a section to $d:V_{\mathbf{\whT}\subset\mathbf{\whG}}\ra \mathbf{\whT}_{\ad}^+$. Its restriction $\fs|_{\mathbf{\whT}_{\ad}}: \mathbf{\whT}_{\ad}\ra V_{\mathbf{\whT}\subset\mathbf{\whG}}|_{\mathbf{\whT}_{\ad}}=\mathbf{\whT}\times^{Z_{\mathbf{\whG}}} \mathbf{\whT}$ is the diagonal embedding $t\mod Z_{\mathbf{\whG}}\mapsto [t,t].$

Let $i_1:\mathbf{\whT}\ra V_{\mathbf{\whT}\subset\mathbf{\whG}}$ be the closed immersion of  $\mathbf{\whT}$ at $d=1$. Since $V_{\mathbf{\whT}\subset\mathbf{\whG}}$ is commutative, the morphism of $\mathbf{\whT}_{\ad}^+$-schemes
 $$
 \xymatrix{
 i_1\cdot\fs:\mathbf{\whT}\times\mathbf{\whT}_{\ad}^+ \ar[r] & V_{\mathbf{\whT}\subset\mathbf{\whG}}
 }
 $$
 is a monoid homomorphism. Its restriction to $\mathbf{\whT}_{\ad}$ is the group isomorphism
  $$
 \xymatrix{
 i_1\cdot(\fs|_{\mathbf{\whT}_{\ad}}):\mathbf{\whT}\times\mathbf{\whT}_{\ad} \ar[r]^<<<<<{\sim} & V_{\mathbf{\whT}\subset\mathbf{\whG}}|_{\mathbf{\whT}_{\ad}}=\mathbf{\whT}\times^{Z_{\mathbf{\whG}}} \mathbf{\whT}.
 }
 $$
\end{Pt*}

\begin{Def*}
The \emph{Vinberg section} is the composition of $\fs: \mathbf{\whT}_{\ad}^+\ra V_{\mathbf{\whT}\subset\mathbf{\whG}}$ with the canonical inclusion $V_{\mathbf{\whT}\subset\mathbf{\whG}}\subset V_{\mathbf{\whG}}$; it is a section to $d:V_{\mathbf{\whG}}\ra \mathbf{\whT}_{\ad}^+$, still denoted by $\fs$:
$$
\xymatrix{
V_{\mathbf{\whG}} \ar[d]^d \\
\mathbf{\whT}_{\ad}^+. \ar@/^1pc/[u]^{\fs}
}
$$
\end{Def*}
\begin{Pt*}
The restriction $\fs|_{\mathbf{\whT}_{\ad}}: \mathbf{\whT}_{\ad}\ra V_{\mathbf{\whG}}|_{\mathbf{\whT}_{\ad}}=\mathbf{\whG}\times^{Z_{\mathbf{\whG}}} \mathbf{\whT}$ splits the exact sequence of group schemes
$$
\xymatrix{
1 \ar[r] & \mathbf{\whG} \ar[r] &  \mathbf{\whG}\times^{Z_{\mathbf{\whG}}} \mathbf{\whT} \ar[r]^>>>>>{\pr_2} &  \mathbf{\whT}_{\ad} \ar[r] & 1.
}
$$
The resulting action of $\mathbf{\whT}_{\ad}$ on $\mathbf{\whG}=\mathbf{\whG}\times\{1_{\mathbf{\whT}}\}\subset \mathbf{\whG}\times^{Z_{\mathbf{\whG}}} \mathbf{\whT}$, equal to $\fs|_{\mathbf{\whT}_{\ad}}$ followed by conjugation, coincides with the faithful adjoint action $\Ad: \mathbf{\whT}_{\ad} \ra \underline{\Aut}(\mathbf{\whG})$. Thus
$$
 \mathbf{\whG}\times^{Z_{\mathbf{\whG}}} \mathbf{\whT}= \mathbf{\whG}\rtimes_{\fs} \mathbf{\whT}_{\ad}=\mathbf{\whG}\rtimes_{\Ad} \mathbf{\whT}_{\ad}=:\mathbf{\whG}\rtimes \mathbf{\whT}_{\ad}.
$$

Pulling-back along $\rho_{\ad}:\bbG_m\ra \mathbf{\whT}_{\ad}$,  we get, cf. \ref{rhoad+}-\ref{Zmon}:
$$
V_{\mathbf{\whG},\rho_{\ad}}|_{\bbG_m}= \mathbf{\whG}\rtimes_{\fs\rho_{\ad}} \bbG_m=\mathbf{\whG}\rtimes_{\Ad\rho_{\ad}} \bbG_m=:\mathbf{\whG}\rtimes \bbG_m =:{}^C\mathbf{G},
$$
i.e. the restriction of the Zhu monoid to $\bbG_m\subset \bbA^1$ is the so-called \emph{$C$-group}. Then diagram \ref{VGcartsquares} specializes as
$$
\xymatrix{
\mathbf{\whG} \ar[r] \ar[d]  &\ar[r] \ar[d]{}^C\mathbf{G} \ar[r] & V_{\mathbf{\whG},\rho_{\ad}}\ar[r] \ar[d]^{d\rho_{\ad}} & \As_{\mathbf{\whG}} \ar[d] \\
\Spec(\bbZ) \ar[r]^1 & \bbG_m \ar[r] &  \bbA^1 & \Spec(\bbZ) \ar[l]_0.
}
$$
\end{Pt*}

\begin{Pt*}\label{rtimes}
Restricting to the diagonalizable submonoid $V_{\mathbf{\whT}\subset\mathbf{\whG},\rho_{\ad}}\subset V_{\mathbf{\whG},\rho_{\ad}}$ \ref{VTGrhoad},
we get
$$
V_{\mathbf{\whT}\subset\mathbf{\whG},\rho_{\ad}}|_{\bbG_m}=\mathbf{\whT}\rtimes_{\fs\rho_{\ad}}\bbG_m=\mathbf{\whT}\rtimes_{\Ad\rho_{\ad}}\bbG_m=:\mathbf{\whT}\rtimes\bbG_m.
$$
Note that the action $\Ad: \mathbf{\whT}_{\ad} \ra \underline{\Aut}(\mathbf{\whT})$ is trivial, so that the semi-direct product $\mathbf{\whT}\rtimes\bbG_m$ is actually direct; it is the isomorphic image of $\mathbf{\whT}\times\bbG_m$ by the monoid homomorphism of $\bbA^1$-schemes
 $$
 \xymatrix{
 i_1\cdot\fs\rho_{\ad}^+:\mathbf{\whT}\times\bbA_1 \ar[r] & V_{\mathbf{\whT}\subset\mathbf{\whG},\rho_{\ad}}.
 }
 $$

Let us fix once and for all a variable $\bfq$ on $\bbA^1$, i.e. $\bbA^1=\Spec(\bbZ[\bfq]).$ \end{Pt*}

\begin{Lem*} \label{ringVTrhoad}
Pull-back of functions by $ i_1\cdot\fs\rho_{\ad}^+$ is the $\bbZ[\bfq]$-algebra embedding 
\begin{eqnarray*}
(i_1\cdot\fs\rho_{\ad}^+)^*:\bbZ[V_{\mathbf{\whT}\subset\mathbf{\whG},\rho_{\ad}}] &\lra &\bbZ[\mathbf{\whT}\times\bbA^1]\\
(e_1^{\nu_1}\otimes e_2^{\nu_2})|_{\rho_{\ad}} & \mapsto & e^{\nu_1}\otimes\bfq^{\lan\nu_1+\nu_2,\rho_{\ad}\ran}.
\end{eqnarray*}
Its image is the subring
$$
\bigoplus_{\nu\in X^*(\mathbf{\whT})}e^{\nu}\otimes \bfq^{\langle \nu-\nu_-,\rho_{\ad}\rangle}\bbZ[\bfq].
$$
\end{Lem*}


\begin{proof}
By definition, pull-back of functions by $ i_1\cdot\fs$ is the $\bbZ[\mathbf{\whT}_{\ad}^+]$-algebra embedding
\begin{eqnarray*}
(i_1\cdot\fs)^*:\bbZ[V_{\mathbf{\whT}\subset\mathbf{\whG}}] &\hra & \bbZ[\mathbf{\whT}]\otimes \bbZ[\mathbf{\whT}_{\ad}^+] \\
e_1^{\nu_1}\otimes e_2^{\nu_2} & \mapsto & e^{\nu_1}\otimes\bar{e}^{\nu_1+\nu_2}.
\end{eqnarray*}
Base changing along 
\begin{eqnarray*}
(\rho_{\ad}^+)^*: \bbZ[\mathbf{\whT}_{\ad}^+] &\lra & \bbZ[\bfq] \\
\bar{e}^{\mu} & \lmapsto & \bfq^{\lan\mu,\rho_{\ad}\ran},
\end{eqnarray*}
we get the formula for $(i_1\cdot\fs\rho_{\ad}^+)^*$. By the above, the image of $(i_1\cdot\fs)^*$ equals the subring
$$
\bigoplus_{\substack{(\nu_1,\mu)\in X^*(\mathbf{\whT})\times X^*(\mathbf{\whT}^+_{\ad})\\ \mu=\nu_1+\nu_2, \nu_i\in X^*(\mathbf{\whT}), \nu_2+(\nu_1)_-\in X^*(\mathbf{\whT}_{\ad})_{\pos}}}
\bbZ(e^{\nu_1}\otimes \bar{e}^{\mu}),
$$
which may be rewritten as 
$$
\bigoplus_{\substack{(\nu_1,\mu)\in X^*(\mathbf{\whT})\times X^*(\mathbf{\whT}^+_{\ad})\\ \mu-\nu_1+(\nu_1)_-\in X^*(\mathbf{\whT}_{\ad})_{\pos}}}
\bbZ(e^{\nu_1}\otimes \bar{e}^{\mu}).
$$
By right-exactness of the tensor product, its base change along $(\rho_{\ad}^+)^*$ equals the image of 
$(i_1\cdot\fs\rho_{\ad}^+)^*$. Hence $(i_1\cdot\fs\rho_{\ad}^+)^*$ identifies $\bbZ[V_{\mathbf{\whT}\subset\mathbf{\whG},\rho_{\ad}}]$ with the subring
$$
\bigoplus_{\substack{(\nu_1,\mu)\in X^*(\mathbf{\whT})\times X^*(\mathbf{\whT}^+_{\ad})\\ \mu-\nu_1+(\nu_1)_-\in X^*(\mathbf{\whT}_{\ad})_{\pos}}}
\bbZ(e^{\nu_1}\otimes \bfq^{\langle \mu,\rho_{\ad}\rangle }).
$$
Note here that for fixed $\nu_1$, the $\mu$ run through the subset 
$\nu_1-(\nu_1)_-+X^*(\mathbf{\whT}_{\ad})_{\pos}$ of 
$X^*(\mathbf{\whT}_{\ad})_{\pos}$ (since $\nu_1\geq (\nu_1)_-$), in particular
 $\langle \mu,\rho_{\ad}\rangle \in\bbN$. The latter subring therefore identifies with 
$$
\bigoplus_{(\nu_1,n)}
\bbZ(e^{\nu_1}\otimes \bfq^n)
$$
where the sum is taken over the $(\nu_1,n)\in X^*(\mathbf{\whT})\times \bbN$ with $n$ in the image of
\begin{eqnarray*}
\nu_1-(\nu_1)_-+X^*(\mathbf{\whT}_{\ad})_{\pos} &\lra & \bbN \\
\mu & \lmapsto & \lan\mu,\rho_{\ad}\ran.
\end{eqnarray*}  
The latter is precisely $\lan\nu_1-(\nu_1)_-,\rho_{\ad}\ran+\bbN$.
\end{proof}


By the preceding lemma \ref{ringVTrhoad}, we have the isomorphism of $\bbZ[\bfq]$-algebras
\begin{eqnarray*}
\bbZ[V_{\mathbf{\whT}\subset\mathbf{\whG},\rho_{\ad}}] &\iso &\bigoplus_{\nu\in X^*(\mathbf{\whT})}e^{\nu}\otimes\bfq^{\lan\nu-\nu_-,\rho_{\ad}\ran}\bbZ[\bfq]\subset \bbZ[\mathbf{\whT}]\otimes\bbZ[\bfq]\\
(e_1^{\nu}\otimes e_2^{-\nu_-})|_{\rho_{\ad}} & \lmapsto & e^{\nu}\otimes\bfq^{\lan\nu-\nu_-,\rho_{\ad}\ran}.
\end{eqnarray*}
\begin{Lem*}\label{lem_q_multiplication}
One has 
$$
(e^{\nu}\otimes \bfq^{\langle \nu-\nu_-,\rho_{\ad}\rangle})(e^{\nu'}\otimes \bfq^{\langle \nu'-(\nu')_-,\rho_{\ad}\rangle}) =\bfq^{\frac{\ell(\nu)+\ell(\nu')-\ell(\nu+\nu')}{2}}(e^{\nu+\nu'}\otimes \bfq^{\langle \nu+\nu'-(\nu+\nu')_-,\rho_{\ad}\rangle}).
$$
\end{Lem*}
\begin{proof}
One has 
$$
\ell(\nu)+\ell(\nu')-\ell(\nu+\nu')=-\lan \nu_{-},2\rho\ran-\lan (\nu')_{-},2\rho\ran+\lan (\nu+\nu')_{-},2\rho\ran,
$$
hence
\begin{eqnarray*}
&&\ell(\nu)+\ell(\nu')-\ell(\nu+\nu')+2\langle \nu+\nu'-(\nu+\nu')_-,\rho_{\ad}\rangle\\
&=&-\lan \nu_{-}+(\nu')_{-}-(\nu+\nu')_{-}-\nu-\nu'+(\nu+\nu')_-,2\rho\ran\\
&=&-\lan \nu_{-}+(\nu')_{-}-\nu-\nu',2\rho\ran\\
&=&2\lan \nu-\nu_-,\rho_{\ad}\ran +2\lan \nu'-(\nu')_-,\rho_{\ad}\ran
\end{eqnarray*}
as claimed. 
\end{proof}
The preceding lemma \ref{lem_q_multiplication} has the following corollary.
Put $$ E(\nu):= (e_1^{\nu}\otimes e_2^{-\nu_-})|_{\rho_{\ad}}.$$ 
\begin{Cor*}\label{Cor_basis} The elements $E(\nu)$ form
a $\bbZ[\bfq]$-basis of $\bbZ[V_{\mathbf{\whT}\subset\mathbf{\whG},\rho_{\ad}}]$, i.e.
$$\bbZ[V_{\mathbf{\whT}\subset\mathbf{\whG},\rho_{\ad}}]=
\bigoplus_{\nu\in X^*(\mathbf{\whT})}\bbZ[\bfq] E(\nu).$$
Moreover, $$
E(\nu)E(\nu')=\bfq_{\nu,\nu'}E(\nu+\nu'),
$$
with $\bfq_{\nu,\nu'}=\bfq^{\frac{\ell(\nu)+\ell(\nu')-\ell(\nu+\nu')}{2}}$
for all $\nu,\nu'\in X^*(\mathbf{\whT}).$ 
\end{Cor*}

\subsection{The action of the Weyl group}

\begin{Pt*}\label{defWfaction}
Using the multiplication law on $V_{\mathbf{\whG}}$, the group $\mathbf{\whG}=V_{\mathbf{\whG},1}$ acts by conjugation on $V_{\mathbf{\whG}}$, and the monoid homomorphism $d:V_{\mathbf{\whG}}\ra \mathbf{\whT}_{\ad}^+$ is equivariant for the trivial $\mathbf{\whG}$-action on the base. Over 
$\mathbf{\whT}_{\ad}\subset \mathbf{\whT}_{\ad}^+$, it is the action by conjugation on the first factor of $V_{\mathbf{\whG}}|_{\mathbf{\whT}_{\ad}}=\mathbf{\whG}\times^{Z_{\mathbf{\whG}}}\mathbf{\whT}$. In particular the normalizer of $\mathbf{\whT}$ stabilizes $\mathbf{\whT}\times^{Z_{\mathbf{\whG}}}\mathbf{\whT}=V_{\mathbf{\whT\subset\whG}}|_{\mathbf{\whT}_{\ad}}$, hence $V_{\mathbf{\whT\subset\whG}}$, and its action factors through $W$. Base changing along $\rho_{\ad}^+$, we get an action of $W$ on $V_{\mathbf{\whT}\subset\mathbf{\whG},\rho_{\ad}}$ for which $d\rho_{\ad}:V_{\mathbf{\whT}\subset\mathbf{\whG},\rho_{\ad}}\ra \bbA^1$ is equivariant for the trivial $W$-action on the base. \end{Pt*}

\begin{Def*}
The action of $W$ on $\mathbf{\whT}\times\bbG_m$ such that the group isomorphism of $\bbG_m$-schemes
$$
 \xymatrix{
 i_1\cdot\fs\rho_{\ad}:\mathbf{\whT}\times\bbG_m \ar[r]^<<<<<{\sim} & V_{\mathbf{\whT}\subset\mathbf{\whG},\rho_{\ad}}|_{\bbG_m}
 }
 $$
 is $W$-equivariant is called the \emph{$W\bullet$-action}. 
\end{Def*}

\begin{Lem*} \label{Wfbulletinvar}
On the ring $\bbZ[\mathbf{\whT}\times\bbG_m]=\bbZ[\mathbf{\whT}]\otimes\bbZ[\bfq^{\pm1}]$, the $W\bullet$-action is:
$$
 \forall w\in W,\quad \quad w\bullet e^{\nu}\bfq^i=e^{w(\nu)}\bfq^{i+\lan w(\nu)-\nu,\rho_{\ad}\ran}.
$$
The subring of invariants is:
$$
\bbZ[\mathbf{\whT}\times\bbG_m]^{W\bullet}=\bigoplus_{\lambda\in X^*(\mathbf{\whT})^+}\bbZ[\bfq^{\pm1}]\sym_{\lambda}(\bfq),
$$
where 
$$
\forall \lambda\in X^*(\mathbf{\whT})^+,\quad\sym_{\lambda}(\bfq):=\sum_{\mu\in W(\lambda)}\bfq^{\lan\rho_{\ad},\mu-w_0(\lambda)\ran}e^{\mu}\quad \in \bbZ[\mathbf{\whT}]\otimes\bbZ[\bfq]=\bbZ[\mathbf{\whT}\times\bbA^1].
$$
In particular,
$$
\bbZ[\mathbf{\whT}\times\bbG_m]^{W\bullet}\cap\bbZ[\mathbf{\whT}\times\bbA^1]=\bigoplus_{\lambda\in X^*(\mathbf{\whT})^+}\bbZ[\bfq]\sym_{\lambda}(\bfq)=:\Sym(\bfq).
$$
\end{Lem*}

\begin{proof}
Let $w\in W$. By \ref{ringVTrhoad}, we have
$$
w\bullet(e^{\nu_1}\otimes\bfq^{\lan\nu_1+\nu_2,\rho_{\ad}\ran})=e^{w(\nu_1)}\otimes\bfq^{\lan w(\nu_1)+\nu_2,\rho_{\ad}\ran},
$$
hence
$$
w\bullet e^{\nu_1}=e^{w(\nu_1)}\otimes\bfq^{\lan w(\nu_1)-\nu_1,\rho_{\ad}\ran}.
$$
The rest of the lemma follows (in particular, note that the coefficient of $e^{w_0(\lambda)}$ in $\sym_{\lambda}(\bfq)$ is $1$).
\end{proof}



\begin{Pt*} For later purposes, we point out that 
the $W$-action on $V_{\mathbf{\whT}\subset\mathbf{\whG}, \rho_{\ad}}$ gives rise to the 
$\bbA^1$-scheme
$$
V_{\mathbf{\whT}\subset\mathbf{\whG},\rho_{\ad}}/W.
$$

\end{Pt*}

\subsection{The geometry of the semigroup $V_{\mathbf{\whT}\subset\mathbf{\whG},0}$}

In this subsection, we describe the geometry of the $0$-fibre $V_{\mathbf{\whT}\subset\mathbf{\whG},0}.$ Since the fibration $d:V_{\mathbf{\whT}\subset\mathbf{\whG}}\ra \mathbf{\whT}_{\ad}^+$ is a monoid homomorphism, the $0$-fiber $V_{\mathbf{\whT}\subset\mathbf{\whG},0}$ is a semigroup scheme over $\bbZ$, whose comultiplication is the specialization at $0$ of the one of the diagonalizable 
$\bbZ$-monoid scheme $V_{\mathbf{\whT}\subset\mathbf{\whG}}$.
\begin{Pt*}
For any Weyl chamber $\mathfrak{C}$ in $X^*(\mathbf{\whT})\otimes\bbR$, write $$X^*(\mathbf{\whT})^\fC:=X^*(\mathbf{\whT})\cap {\overline{\fC}}\hskip20pt \text{and}\hskip20pt \dsU_{\mathbf{\whG}}^{\fC}:=\Spec(\bbZ[X^*(\mathbf{\whT})^\fC]).$$
The affine variety $\dsU_{\mathbf{\whG}}^{\fC}$ is normal and contains a copy of $\Spec\bbZ$
given by the ring homomorphism $\bbZ[X^*(\mathbf{\whT})^\fC])\rightarrow\bbZ, e^\nu\mapsto 0$.

\vskip5pt 
Recall from \ref{Cor_basis} the $\bbZ[\bfq]$-basis $E(\nu)$ of $\bbZ[V_{\mathbf{\whT}\subset\mathbf{\whG}}]$ and the product formula 
$
E(\nu)E(\nu')=\bfq_{\nu,\nu'}E(\nu+\nu'),
$
with $\bfq_{\nu,\nu'}=\bfq^{\frac{\ell(\nu)+\ell(\nu')-\ell(\nu+\nu')}{2}}$. The latter quantity is equal to $1$ if there exists a Weyl chamber $\fC$ such that $\nu$ and $\nu'$ belong to the closure $\overline{\fC}$, and it is divisible by $\bfq$ otherwise. 
\vskip5pt

For simplicity, we use the same notation for the images of the $E(\nu)$ in the quotient ring $\bbZ[V_{\mathbf{\whT}\subset\mathbf{\whG},0}]$, i.e. we set 
$$
\forall\nu\in X^*(\mathbf{\whT}),\quad E(\nu):=(e_1^{\nu}\otimes e_2^{-\nu_-})|_{0}\in \bbZ[V_{\mathbf{\whT}\subset\mathbf{\whG},0}].
$$
The following proposition is then immediate.
\end{Pt*}
\begin{Prop*}\label{Prop_basis2} The elements $E(\nu)$ form
a $\bbZ$-basis of $\bbZ[V_{\mathbf{\whT}\subset\mathbf{\whG},0}]$, i.e.

$$
\bbZ[V_{\mathbf{\whT}\subset\mathbf{\whG},0}]=\bigoplus_{\nu\in X^*(\mathbf{\whT})}\bbZ E(\nu),
$$
with product formula 
$$
\forall \nu,\nu'\in X^*(\mathbf{\whT}), \quad E(\nu)E(\nu')=
\left\{ \begin{array}{ll}
E(\nu+\nu') & \textrm{if $\exists\fC: \nu,\nu' \in X^*(\mathbf{\whT})^\fC$}\\ 
0 & \textrm{otherwise}.
\end{array} \right.
$$
\end{Prop*}

\begin{Lem*}
The additive map
\begin{eqnarray*}
\pr_\fC: \bbZ[V_{\mathbf{\whT}\subset\mathbf{\whG},0}] & \lra & \bbZ[\dsU_{\mathbf{\whG}}^\fC] \\
E(\nu) & \lmapsto & \left\{ \begin{array}{ll}
e^{\nu} & \textrm{if $\nu\in X^*(\mathbf{\whT})^\fC$}\\ 
0 & \textrm{otherwise}
\end{array} \right.
\end{eqnarray*}
is a surjective homomorphism of rings. Hence, it defines a closed immersion 
$ i_{\fC}: \dsU_{\mathbf{\whG}}^{\fC}\longrightarrow V_{\mathbf{\whT}\subset\mathbf{\whG},0}.$

\end{Lem*}
\begin{proof} The compatibility with the ring multiplication is clear from the product formula 
of the $E(\nu)$. Surjectivity is clear. 
\end{proof}

\begin{Prop*} \begin{enumerate}
\item The affine variety $V_{\mathbf{\whT}\subset\mathbf{\whG},0}$ is reduced. 
\item $\cap_{\fC}  \dsU_{\mathbf{\whG}}^{\fC}=\Spec\bbZ.$
\item $\cup_{\fC}  \dsU_{\mathbf{\whG}}^{\fC}= V_{\mathbf{\whT}\subset\mathbf{\whG},0}.$
\item The $\dsU_{\mathbf{\whG}}^{\fC}$ are the irreducible components of $V_{\mathbf{\whT}\subset\mathbf{\whG},0}$.

\end{enumerate}
\end{Prop*} 
\begin{proof} 1. All the rings $\bbZ[X^*(\mathbf{\whT})^\fC]$ are integral domains. It follows from the elementary fact $\cup \fC= X^*(\mathbf{\whT})$ that the product map 
$\bbZ[V_{\mathbf{\whT}\subset\mathbf{\whG},0}]
\rightarrow \prod_{\fC} \bbZ[X^*(\mathbf{\whT})^\fC]$ is injective. Hence the source is reduced. 2. follows from the fact that $\cap\fC=\{0\}$. For 3. let $\mathfrak{m}\subseteq V_{\mathbf{\whT}\subset\mathbf{\whG},0}$ be a maximal ideal. We have to show $\ker\pr_{\fC}\subseteq \mathfrak{m}$ for some $\fC$ or, equivalently, $\bbZ[V_{\mathbf{\whT}\subset\mathbf{\whG},0}]\setminus \mathfrak{m} \subseteq \bbZ[V_{\mathbf{\whT}\subset\mathbf{\whG},0}] \setminus \ker\pr_{\fC}$.
We claim that it suffices to see that the set $S:=\{\nu: E(\nu)\notin\mathfrak{m}\}$ is contained in some $\overline{\fC}$. Assume this is the case and let $E\in\bbZ[V_{\mathbf{\whT}\subset\mathbf{\whG},0}]\setminus \mathfrak{m}$. Write $E=\sum_\nu a_\nu E(\nu)$ with coefficients $a_\nu\in\bbZ$. Then $a_\nu\neq 0$ for some $\nu\in S$ and therefore $\pr_{\fC}(E)\neq 0$. To verify the claim, it suffices to see that for any coroot $\alpha$ of $(\mathbf{\whG},\mathbf{\whT})$ we either have 
$\langle S,\alpha\rangle \leq 0$ or $\langle S,\alpha\rangle \geq 0$. Given $\nu,\nu'\in S$, we have
$E(\nu)E(\nu')\notin\mathfrak{m}$. This implies $E(\nu)E(\nu')\neq 0$ and so $\nu,\nu'$ must lie in the closure of some common Weyl chamber. In particular, for any $\alpha$ the two values $\langle \nu,\alpha\rangle$ and $\langle \nu',\alpha\rangle$ both are nonnegative or nonpositive. Now fix $\alpha$ and assume that $\langle S,\alpha\rangle\neq \{0\}$. Choose $\nu\in S$ such that  $\langle \nu,\alpha\rangle<0$ or $\langle \nu,\alpha\rangle>0$. For any other $\nu'\in S$ the previous discussion shows that in the first case $\langle \nu',\alpha\rangle\leq 0$ and in the second case $\langle \nu',\alpha\rangle\geq 0$. Hence in the first case $\langle S,\alpha\rangle \leq 0$ and in the second case $\langle S,\alpha\rangle \geq 0$. 4. follows from 3. together with [B-CA], II. 4.1 Prop. 6.
\end{proof}
It follows from the proposition that the $W$-action on $V_{\mathbf{\whT}\subset\mathbf{\whG},0}$ permuts its irreducible components in a simply transitive manner. 
\begin{Lem*} \label{lem-central}
Let $\fC$ be a Weyl chamber. For any $x\in X^*(\hT)^\fC$, let 
$$Z(x):=\sum_{ x'\in W_{\fin}.x} E(x').$$ 
The $Z(x)$ form a $\bbZ$-basis of $\bbZ[V_{\hT\subset\hG,0}]^{W}$. 
\end{Lem*} 
\begin{proof}
According to the above proposition $V_{\hT\subset\hG,0}=\cup_{\fC} \dsU_{\mathbf{\whG}}^{\fC}$
is its decomposition into irreducible components, where $\dsU_{\mathbf{\whG}}^{\fC}
=\Spec \bbZ[X^*(\hT)^\fC ]$.
Now $X^*(\hT)^\fC$ is a fundamental domain for the action of $W$ on 
$X^*(\hT)$, so the $Z(x)$ are linearly independent elements. They are clearly $W$-invariant. Conversely, let $y\in \bbZ[V_{\hT\subset\hG,0}]^{W}$, say 
$y=\sum_{\nu} a_{\nu} E(\nu)$. Given $w\in W$, then $y=w(y)=\sum_{\nu} a_{\nu} E(w\nu)$ which implies 
$a_{\nu}=a_{w\nu}$ for all $\nu$. This implies that $y$ lies in the span of the $Z(x)$. 
\end{proof}

\begin{Cor*}\label{sectiontoGIT}
Let $\fC$ be a Weyl chamber. The composed morphism of $\bbZ$-schemes
$$
\xymatrix{
\dsU_{\mathbf{\whG}}^{\fC} \ar@{^{(}->}[r] & V_{\mathbf{\whT}\subset\mathbf{\whG},0} \ar@{->>}[r] & V_{\mathbf{\whT}\subset\mathbf{\whG},0}/W
}
$$
is an isomorphism.
\end{Cor*}
\begin{proof} The corresponding algebra homomorphism maps $Z(x)$ to $E(x)$, hence is bijective.
\end{proof}

\begin{Pt*}
Recall the Vinberg section
 $$\fs\rho_{\ad}^+: \bbA_1 \rightarrow V_{\mathbf{\whT}\subset\mathbf{\whG},\rho_{\ad}}.$$
The point $\fs\rho_{\ad}^+(0)$ lies on some irreducible component of $V_{\mathbf{\whT}\subset\mathbf{\whG},0}$, which we view thereby as a "marked" component. To make this precise, let us consider the antidominant chamber $\fC^{-}$ and abbreviate 
$$\dsU_{\mathbf{\whG}}:= \dsU_{\mathbf{\whG}}^{\fC^{-}}=\bbZ[X^*(\mathbf{\whT})^{-}].$$
 \end{Pt*}
\begin{Prop*} The unit of the monoid $\dsU_{\mathbf{\whG}}$ is mapped via $i_{\fC^{-}}$ to  $\fs\rho_{\ad}^+(0)$.
\end{Prop*}
\begin{proof}
The unit of the monoid $\dsU_{\mathbf{\whG}}$ is given by the ring homomorphism
$\bbZ[\dsU_{\mathbf{\whG}}]\rightarrow \bbZ, e^\nu\mapsto 1$. In turn, $\fs\rho_{\ad}^+(0)$
is given by the composition of the homomorphism 
$\bbZ[V_{\mathbf{\whT}\subset\mathbf{\whG},0}]\rightarrow \bbZ[\bfq], E(\nu)\mapsto \bfq^{\langle \nu-\nu_{-},\rho_{\ad}\rangle}$ together with $\bfq=0$. So we are reduced to show that the latter composite map sends $E(\nu)$ to $1$ whenever $\nu\in X^*(\mathbf{\whT})^-$. But this is clear, since $\nu=\nu_{-}$ in this case. 
\end{proof}

\begin{Pt*} 
We refer to the appendix for some basic information and notation on the scheme $\dsU_{\mathbf{\whG}}$.
In particular, it is toric scheme over $\bbZ$ for the torus $\hT$. Its stratification by $\hT$-orbits 
$\dsS_{\mathbf{\whL}}$ is naturally indexed by the standard Levi subgroups $\mathbf{\whL}\in\widehat{\cL}$ in $\hG$. Each orbit $\dsS_{\mathbf{\whL}}$ is a torus and contains a distinguished point $e_{\hL}$, the multiplicative unit element of the torus. The resulting map 
$\hT\rightarrow \dsS_{\mathbf{\whL}}, x\mapsto x.e_{\hL}$ fits into 
the short exact sequence 
$$ 1 \longrightarrow \prod_{I} \bbG_m\stackrel{\prod_I \alpha}{\longrightarrow}\hT\longrightarrow \dsS_{\mathbf{\whL}}\rightarrow 1.$$ 
Here, $\hL=\hL(\widehat{I})$ with corresponding set of coroots $\alpha\in I \subset X_*(\hT)$. Restricting characters of the torus $\hL^{\ab}=\hL/\hL^{\der}$ along the map $\hT\rightarrow \hL^{\ab}$ yields a canonical isomorphism of tori $\hL^{\ab}\simeq  \dsS_{\mathbf{\whL}}$. We use this as an identification in what follows, in particular we write 
$$ \dsU_{\hG}=\bigcup_{\hL\in\widehat{\cL}} \hL^{\ab}.$$

\end{Pt*}
\begin{Pt*} \label{Pt_par_stab}
Recall that $W$ acts simply transitive on the set of irreducible components of $V_{\hT\subset\hG,0}$. We put the orbit stratification from the "marked"  component
$\dsU_{\mathbf{\whG}}$ on all other components by setting 
$$V_{\hT\subset\hG,0,\hL}:=\bigcup_{w\in W} w\hL^{\ab}.$$
One obtains a stratification into locally closed subvarieties  $$V_{\hT\subset\hG,0}=\bigcup_{\hL\in\widehat{\cL}} 
V_{\hT\subset\hG,0,\hL}.$$ By construction, each stratum 
$V_{\hT\subset\hG,0,\hL}$ is $W$-stable. 

\vskip5pt 
Let $W_\hL\subset W$ be the parabolic subgroup of $W$ associated with $\hL$. 
Let $w\in W$. 
Then $$w\hL^{\ab}=\hL^{\ab} \hskip10pt \text{if and only if}\hskip10pt w\in W_\hL.$$
Indeed, if $\hL=\hL(\widehat{I})$, then $w\hL^{\ab}=\hL^{\ab}$ if and only if $w$ stabilizes the simplex of the Coxeter complex of $W$ defined by the subset $\widehat{I}\subset\widehat{\Delta}$.
\end{Pt*}

\subsection{The augmented semigroup $V^{(1)}_{\mathbf{\whT}\subset\mathbf{\whG},0,\bbF_q}$ over a finite field}\label{subsec_augmented_semi}
Let $q$ be a power of a prime $p$. Denote by $\bbF_q$ the finite field with $q$ elements.\begin{Pt*} Denote by $\hT_{\bbF_q}$ the base change of $\hT$ to $\bbF_q$.
Let $F_q$ be the $q$-Frobenius $x\mapsto x^q$ on $\hT_{\bbF_q}$ and let 
$(\hT_{\bbF_q})^{F_q}$ be its fixed points. As as set $(\hT_{\bbF_q})^{F_q}=\hT(\bbF_q)$, the set of $\bbF_q$-valued points of $\hT$, and we will use this notation from now on. 
Note that $\hT(\bbF_q)$ is a finite diagonalizable group scheme over $\bbF_q$
with group of characters $(\bbZ/(q-1)\bbZ)^r$, where $r$ is the rang of $\hT$. Since $q-1\in\bbF_q^\times$, the group scheme is étale and even constant. 
Since $W$ commutes with $F_q$, $\hT(\bbF_q)$ is stable under the $W$-action. 
\end{Pt*}

\begin{Pt*} We denote by $$V_{\hT\subset\hG,0,\bbF_q}:=V_{\hT\subset\hG,0}\otimes_{\Spec\bbZ} \Spec(\bbF_q)$$
the base change of $V_{\hT\subset\hG,0}$ to $\bbF_q$.
Consider the augmented semigroup over $\bbF_q$
$$V^{(1)}_{\hT\subset\hG,0,\bbF_q}:=\hT(\bbF_q)\times_{\bbF_q} V_{\hT\subset\hG,0,\bbF_q}=
\hT(\bbF_q)\times_{\bbZ} V_{\hT\subset\hG,0}.
$$
It has its diagonal $W$-action. The corresponding quotient 
$$ S:= V^{(1)}_{\hT\subset\hG,0,\bbF_q}/W$$
generalizes the Satake scheme, denoted $S(q)$, in the $GL_2$-case from \cite[Def. 5.2]{PS25}.
\end{Pt*}
\begin{Pt*}\label{Pt_stratification}
Let $\hL\in \widehat{\cL}$ be a standard Levi subgroup in $\hG$. One sets
$$V^{(1)}_{\hT\subset\hG,0,\hL}:=\hT(\bbF_q)\times_{\bbZ} V_{\hT\subset\hG,0,\hL}.$$ This yields a 
$W$-invariant stratification of the $\bbF_q$-scheme $V^{(1)}_{\hT\subset\hG,0,\bbF_q}$
$$V^{(1)}_{\hT\subset\hG,0,\bbF_q}=\bigcup_{\hL\in\widehat{\cL}} 
V^{(1)}_{\hT\subset\hG,0,\hL}.$$ 
\end{Pt*}
\begin{Pt*} 
Let $\hL\in \widehat{\cL}$ be a standard Levi subgroup in $\hG$ with Weyl group $W_\hL\subset W$. Define the quotient scheme 
$$S_\hL:= V^{(1)}_{\hT\subset\hG,0,\hL}/W.$$
Thus,
$$
S =\bigcup_{\hL\in\widehat{\cL}} S_{\hL}.
$$

The scheme $S_{\hL}$ is a disjoint union of copies $\hL^{\ab}$ indexed by the set $\hT(\bbF_q)/W_{\hL}$ of $W_{\hL}$-orbits in $\hT(\bbF_q)$. 
Indeed, one has 
$$V^{(1)}_{\hT\subset\hG,0,\hL}=W.\hL^{\ab}$$
with ${\rm Stab}_W(\hL^{\ab})=W_\hL$, cf. \ref{Pt_par_stab}, 
whence $$S_{\hL}=(\hT(\bbF_q)\times W.\hL^{\ab})/W=\hT(\bbF_q)/W_\hL \times \hL^{\ab}.$$

\end{Pt*}
\section{Tame Galois representations}

Let $p$ be a prime. Fix an algebraic closure $\overline{\bbQ}_p$ of $\bbQ_p$ and denote by $\overline{\bbF}_p$ its residue field. For all $n$, denote by $\bbQ_{p^n}\subset \overline{\bbQ}_p$ the unique subfield that is unramified and of degree $n$ over $\bbQ_p$. Let $\bbF_{p^n}\subset\overline{\bbF}_p$ denote the unique subfield of cardinality $p^n$.

\vskip5pt 
Let $I_p$ be the inertia subgroup of ${\rm Gal}(\overline{\bbQ}_p/ \bbQ_p).$ 
Let $P_p\subset I_p$ be the wild ramification group. Let $q$ some fixed power of $p$.
Write $\cG_q:={\rm Gal}(\overline{\bbQ}_p/ \bbQ_q)$, with inertia subgroup $I_p$. Set $\cG^t_q:=\cG_q/P_p$ and $I_p^t:=I_p/P_p$. By an arithmetic Frobenius element $\varphi\in \cG_q$ we mean an element lifting the Frobenius $x\mapsto x^q$ in 
${\rm Gal}(\overline{\bbF}_p/ \bbF_q)$.

\vskip5pt
To simplify notation, we will write in the following $\mathbf{\whG},\mathbf{\whT}$ etc. for the base change of these $\bbZ$-schemes to the algebraic closure $\overline{\bbF}_p$. To keep the notation simpler, we will confuse these latter algebraic groups with their groups of geometric points 
$\mathbf{\whG}(\overline{\bbF}_p),\mathbf{\whT}(\overline{\bbF}_p)$ etc., the latter being viewed as discrete topological groups. For two topological groups $H_1,H_2$ the notation $\Hom(H_1,H_2)$ denotes the set of ${\it continuous}$ group homomorphisms.

\vskip5pt
For any $w\in W$, we choose a lift $\dot{w}\in\mathbf{\whG}$. To ease notation, we continue to write $w$ for this lift. 
\vskip5pt
Whenever we specialize to the case $\mathbf{\whG}=GL_n$ in the following, we tacitly assume that $\mathbf{\whT}\subset \mathbf{\whB}$ equals the standard torus of diagonal matrices and the Borel subgroup of upper triangular matrices respectively. 
\subsection{Generalities on tame inertial types}

\begin{Def*} A {\it mod $p$ tame inertial type (with values in $\hG$)} is a continuous homomorphism 
$$\tau: I_p\lra \hG$$
that is trivial on $P_p$ and extends to the group $\cG_q$. We denote the set of mod $p$ tame types by $\cT_\hG$. There is an obvious conjugation action of $\hG$ on the set $\cT_\hG$.
\end{Def*}

\begin{Pt*}\label{Pt_can_iso}
There is an isomorphism $$I_p^t\iso \varprojlim_i \bbF^{\times}_{p^{i}},$$ which is canonical by definition of the field $\overline{\bbF}_p$ as residue field of $\overline{\bbQ}_p$, and its finite subfields. We therefore identify the two groups in what follows. In particular, any choice of topological generator $(\zeta_{p^{i}-1})_{i=1}^\infty$ for $\varprojlim_i \bbF^{\times}_{p^{i}}$ identifies a topological generator $M$ for the group $I_p^t$. 
\end{Pt*}
\begin{Pt*}
Let $\tau$ be a mod $p$ tame inertial type. The conjugacy class of $\tau$ is determined by the conjugacy class of the element $\tau(M)$. Since $\tau$ extends to $\cG_q$, this class is stable under the Frobenius $F_q: x\mapsto x^q$. Moreover, any given element $g\in\hG$ has order prime to $p$ if and only if its is semisimple. It follows that 
$\tau\mapsto\tau(M)$ induces a bijection, depending on the choice of $M$, 
\begin{eqnarray*}
\{\textrm{mod $p$ tame inertial types}\}/ \hG & \simeq & 
\{\textrm{$F_q$-stable semisimple conjugacy classes in $\hG$}\}
\end{eqnarray*}
It is well-known that the set on the right hand side is a finite set. 
\end{Pt*}
\begin{Pt*} \label{Pt_types}
Since $F_q$ commutes with $W$, there is a $F_q$-action on the set $\Hom(I_p^t, \hT)/W$.
 The inclusion $\hT\subset\hG$ induces is a canonical map 
$$(\Hom(I_p^t, \hT)/W)^{F_q}\iso \{\textrm{mod $p$ tame inertial types}\}/ \hG. $$
The previous paragraph implies that this is a bijection.
We will identify the two sets from now on via this bijection.
\end{Pt*}
\begin{Pt*}
Let $\tau\in \Hom(I_p, \hT)$ with $\tau^q=w\tau w^{-1}$ for some $w\in W$. Let $\sigma\in W$. Then $$(\sigma \tau\sigma^{-1})^q=\sigma\tau^q\sigma^{-1}= w' (\sigma \tau\sigma^{-1})w'^{-1}$$ where $w'=\sigma w\sigma^{-1}$. Hence if $F_q$ acts on some representative of a given $F_q$-stable class in 
$(\Hom(I_p^t, \hT)/W)^{F_q}$ through a Coxeter element, then $F_q$ acts on any representative of the class through a Coxeter element. 
Denote by 
 $\Hom(I_p^t, \hT)/W)^{F_q,\rm cox}$ the subset of classes on which $F_q$ acts through a Coxeter element of $W$ on some (equivalently, any) representative of the class. In this latter subset, denote by $(\Hom(I_p^t, \hT)/W)^{F_q,\rm reg}$
the subset of regular classes, i.e. which have maximal cardinality equal to $\mid W\mid$. Thus 
$$ (\Hom(I_p^t, \hT)/W)^{F_q,\rm reg}\subset \Hom(I_p^t, \hT)/W)^{F_q,\rm cox}\subset (\Hom^{cts}(I_p^t, \hT)/W)^{F_q}.$$

\end{Pt*}

\subsection{Types with values in Levi subgroups}

Let $\hL\in\widehat{\cL}$ be a standard Levi subgroup of $\hG$.

\begin{Pt*} The discussion of the previous subsection above applies more generally to the reductive group $\hL$, instead of $\hG$. We therefore have the notion of a 
mod $p$ tame inertial type with values in $\hL$
$$\tau: I_p\lra \hL.$$
We denote the set of such types by $\cT_\hL$, with its natural $\hL$-action. 
The quotient $\cT_\hL/\hL$
is in bijection with the set of $F_q$-stable semisimple conjugacy classes in $\hL$, as before. 
\end{Pt*}
\begin{Pt*}
Denote by 
 $\Hom(I_p^t, \hT)/W_{\hL})^{F_q,\rm cox}$ the subset of classes on which $F_q$ acts through a Coxeter element of $W_{\hL}$ on each representative. In this latter subset, denote by $(\Hom(I_p^t, \hT)/W_{\hL})^{F_q,\rm reg}$
the subset of classes which have maximal cardinality $\mid W_{\hL}\mid$. As before,
$$ (\Hom(I_p^t, \hT)/W_{\hL})^{F_q,\rm reg}\subset \Hom(I_p^t, \hT)/W_{\hL})^{F_q,\rm cox}\subset (\Hom^{cts}(I_p^t, \hT)/W_{\hL})^{F_q}.$$

\end{Pt*}

\begin{Pt*}
The finite group $W(\hL):=N_{\hG}(\hL)/\hL$ naturally acts on the set $\cT_{\hL}/\hL.$
The map 
$$ \cT_\hL/\hL \lra \cT_{\hG}/\hG,$$ coming from the inclusion $\hL\subset\hG$, is $W(\hL)$-invariant, i.e. constant on 
$W(\hL)$-orbits.
 \end{Pt*}

\subsection{Tame Galois representations}
\begin{Pt*} Given $g\in\hG$, there is the evaluation homomorphism 
${\rm ev}_{g}\in \Hom(X^*(\hG),\overline{\bbF}_p^\times)$ sending a character 
 $\chi\in X^*(\hG)$ to $\chi(g)$. The map $g\mapsto {\rm ev}_{g}$ is multiplicative. One obtains 
a canonical map 
$$ \Hom(I_p^t, \hT)\lra \Hom (I_p^t,\Hom(X^*(\hG),\overline{\bbF}_p^\times))$$
mapping $\rho$ to $\sigma\mapsto {\rm ev}_{\rho(\sigma)}$. Given $w\in W$, one has 
$${\rm ev}_{w\rho(\sigma)w^{-1}}(\chi)=\chi( w\rho(\sigma)w^{-1}) = \chi(w)\chi( \rho(\sigma) )\chi(w)^{-1}=\chi(\rho(\sigma))={\rm ev}_{\rho(\sigma)}(\chi)$$
for any $\chi\in \Hom(X^*(\hG),\overline{\bbF}_p^\times)$,
since $\overline{\bbF}_p^\times$ is commutative. So the above map is constant on $W$-orbits.  

\vskip5pt 
We view the group $\Hom(X^*(\hG),\overline{\bbF}_p^\times)$ as the group of points of the
torus $\hG^{\rm ab}$, via the canonical isomorphisms
$$\Hom(X^*(\hG),\overline{\bbF}_p^\times)\simeq 
\Hom(X^*(\hG^{\ab}),\overline{\bbF}_p^\times)\simeq \Hom_{\rm alg}(\cO(\hG^{\ab}),\overline{\bbF}_p)=\hG^{\ab}(\overline{\bbF}_p).$$
All in all, there is a canonical map 
$$ \Hom(I_p^t, \hT)\lra \Hom (I_p^t,\hG^{\rm ab})$$
which is constant on $W$-orbits. It therefore gives rise to a map 
$\Hom(I_p^t, \hT)/W\ra  \Hom (I_p^t,\hG^{\rm ab})$ and, by restriction to Frobenius stable classes, to a map 
$$ (\Hom(I_p^t, \hT)/W)^{F_q}\lra  \Hom (I_p^t,\hG^{\rm ab}).$$
Using the inclusion $I_p^t \subset \cG_q^t$, we may form 
the fibre product $$(\Hom(I_p^t, \hT)/W)^{F_q} \; \times_{\Hom (I_p^t,\hG^{\rm ab})} \Hom (\cG_q^t,\hG^{\rm ab}).$$
\end{Pt*}
\begin{Pt*}
According to \ref{Pt_types}, there is a 
canonical map 
$$f_{\hG}: \Hom(\cG_q^t, \hG)/\hG\lra (\Hom(I_p^t, \hT)/W)^{F_q} \; \times_{\Hom (I_p^t,\hG^{\rm ab})} \Hom (\cG_q^t,\hG^{\rm ab}).$$
Note that any pair $(\psi_1,\psi_2)$ in the fibre product on the right-hand side has the property that the image of 
$\psi_1$ in $\Hom (I_p^t,\hG^{\rm ab})$ is $F_q$-invariant, i.e. is a homomorphism of $I_p^t$ into the 
group of $\bbF_q$-rational points $\hG^{\rm ab}(\bbF_q)$. 
\end{Pt*}
\begin{Prop*}\label{prop_Born1} Let $\mathbf{\whG}=GL_n$ and denote by $\Hom(\cG^t_q, \hG)^{\rm irr}$ the set of irreducible $n$-dimensional representations of $\cG_q$.  The above canonical map induces a bijection 
$$\Hom(\cG^t_q, \hG)^{\rm irr}/\hG \iso (\Hom(I_p^t, \hT)/W)^{F_q,\rm reg}  \; \times_{\Hom (I_p^t,\hG^{\rm ab})} \Hom (\cG^t_q,\hG^{\rm ab}).$$
\end{Prop*}
\begin{proof} This is \cite[Thm. 2.1.7]{B15}. Note that loc. cit. uses $\hG^{\rm ab}\simeq \bbG_m$ via the determinant. 
\end{proof}

\begin{Prop*}\label{prop_Born2}  Let $\mathbf{\whG}=GL_n$. Choose a topological generator $M\in I_p^t$ and an arithmetic Frobenius $\varphi\in \cG_q$. There is a map $h_{\hG}^{M,\varphi}$ depending on the pair $(M,\varphi)$
$$\Hom(\cG_q^t, \hG)/\hG \longleftarrow (\Hom(I_p^t, \hT)/W)^{F_q,\rm cox}  \; \times_{\Hom (I_p^t,\hG^{\rm ab})} \Hom (\cG_q^t,\hG^{\rm ab}),$$
which is a section to the canonical map $f_{\hG}$ in the sense that $f_{\hG}\circ h_{\hG}^{M,\varphi}={\rm id}$.
\end{Prop*}
\begin{proof} This is a corollary of the proof of \cite[Thm. 2.1.7]{B15}.
We identify $\hG^{\rm ab}\simeq \bbG_m$ via the determinant. 
Suppose that $(W\psi_1,\psi_2)$ is a point in the fibre product, with a chosen representative $\psi_1$. So $\psi_1(M)=\diag (y_1,...,y_n)\in\hT$. After renumbering the $y_i$, i.e. choosing a different representative of the orbit, we have $y_{i+1}=y_i^q$ for $1\leq i\leq n$ where $y_{n+1}:=y_1$.
Hence, one obtains an element of the form $u=\diag(y,y^q,...,y^{q^n})$, where $y:=y_1\in\bbF_{q^n}$.
If $z:=\det\psi_2(\varphi)$, one defines an element $s\in\hG$ as 
$$s:=\left( \begin{array}{rrrrr}
0 & 1 &  & &\\
 & 0 & 1 & &\\
 &  & \ddots & \ddots & \\
 &  &  &0 &1\\
(-1)^nz&  &  & &0\\
\end{array}\right)$$
Then $sus^{-1}=u^q$. Sending 
$M\mapsto u, \varphi\mapsto s$ defines a tame representation $\rho$ of $\cG_q$. 
We define $h_{\hG}^{M,\varphi}(W\psi_1,\psi_2)$ to be the $\hG$-conjugacy class of $\rho$. 
Choosing a different representative $\psi_1$ transforms $y$ into a $q$-power $y'=y^{q^{i}}$. Then $y$ and $y'$ are 
${\rm Gal}(\bbF_{q^n}/\bbF_q)$-conjugate. According to \cite[Prop. 2.1.6 (ii)]{B15}, the pairs
$(u,s)$ and $(u',s)$ are $\hG$-conjugate. Hence, the map $h_{\hG}^{M,\varphi}$ is well-defined. 
It is a section of $f_{\hG}$ by construction. 
\end{proof} 
\begin{Rem*}\label{Rem_irr} In the notation of the preceding proof: The orbit $W\psi_1$ lies in the smaller set 
 $(\Hom(I_p^t, \hT)/W)^{F_q,\rm reg}$ if and only if all $y_i$ are pairwise different, i.e. if and only if $y=y_1$ does not lie in any proper subfield of $\bbF_{q^n}$.
\end{Rem*}
\subsection{Representations with values in Levi subgroups}
Let $\hL\in\hat{\cL}$ be a standard Levi subgroup of $\hG$.
\begin{Pt*} The discussion of the preceding subsection applies more generally to $\hL$ instead of $\hG$. 
In particular, there is a canonical map 
$$f_{\hL}: \Hom(\cG_q^t, \hL)/\hG\lra (\Hom(I_p^t, \hT)/W_{\hL})^{F_q} \; \times_{\Hom (I_p^t,\hL^{\rm ab})} \Hom (\cG_q^t,\hL^{\rm ab}).$$

Fix a topological generator $M$ of $I_p^t$ and a Frobenius $\varphi$ in $\cG_q$. This defines a splitting 
$\cG_q^t\simeq I_p^t \rtimes \bbZ.$ Since $\hL^{\rm ab}$ is abelian, the splitting defines a bijection 
$$\Hom (\cG_q^t,\hL^{\rm ab})\simeq \{ (u,s)\in \hL^{\rm ab}: u=u^q\}=\hL^{\rm ab}(\bbF_q)\times\hL^{\rm ab}.$$
This means that the fibre product appearing as the target of $f_{\hL}$ can be written as 
$$(\Hom(I_p^t, \hT)/W_{\hL})^{F_q} \; \times_{\Hom (I_p^t,\hL^{\rm ab})} \Hom (\cG_q^t,\hL^{\rm ab})= \Hom(I_p^t, \hT)/W_{\hL})^{F_q}\times \hL^{\rm ab},$$
i.e. has the structure of a finite disjoint union of algebraic tori defined over $\bbF_q$ (namely finitely many copies of  $\hL^{\rm ab}$).
In particular, the map $f_\hL$ can be written as 
$$f_\hL:  \Hom(\cG_q^t, \hL)/\hL\lra (\Hom(I_p^t, \hT)/W_{\hL})^{F_q} \; \times \hL^{\rm ab}.$$
\end{Pt*}
\begin{Def*} \label{Def_YL} Let 
$$ \;\;X_\hL:=(\Hom(I_p^t, \hT)/W_{\hL})^{F_q,\rm cox} \; \times_{\Hom (I_p^t,\hL^{\rm ab})} \Hom (\cG_q^t,\hL^{\rm ab})$$

and 
 $$ X^{\rm reg}_\hL:= (\Hom(I_p^t, \hT)/W_{\hL})^{F_q,\rm reg} \; \times_{\Hom (I_p^t,\hL^{\rm ab})} \Hom (\cG_q^t,\hL^{\rm ab}).$$ 
 \end{Def*}
 As we have seen above, these sets have the structure of a finite union of algebraic tori defined over $\bbF_q$. In particular, they are affine schemes defined over $\bbF_q$. 
\vskip5pt 
As we will recall in \ref{Pt_rappel} below, the finite group
$W(\hL):=N_{\hG}(\hL)/\hL$ identifies with a subgroup of $W$ which commutes with the action of $W_{\hL}$ and of $F_q$.
It therefore acts on $\Hom(I_p^t, \hT)/W_{\hL})^{F_q}$ and its subsets corresponding to Coxeter classes and regular classes. Moreover, the group $W(\hL)$ acts naturally on $\hL^{\rm ab}$ and 
$\Hom(\cG_q^t, \hL)/\hL$. In particular, it acts on the affine schemes $Y_\hL$ and $Y^{\rm reg}_\hL.$

\begin{Def*}\label{Def_X[L]}  \label{Pt_equiv} Put $ X_{[\hL]}:= X_\hL / W(\hL)$ and $X^{\rm reg}_{[\hL]}:= X^{\rm reg}_\hL / W(\hL)$.  

\end{Def*} 
Since $W(\hL)$ is finite, these are naturally affine schemes defined over $\bbF_q$.

\begin{Prop*} \label{Prop_irr} Let $\mathbf{\whG}=GL_n$ and denote by $\Hom(\cG^t_q, \hL)^{\rm irr}$ the set of irreducible representations of $\cG_q$ with values in $\hL$. 
The canonical map $f_\hL$ induces a bijection 
$$\Hom(\cG_q^t, \hL)^{\rm irr}/\hL \iso  (\Hom(I_p^t, \hT)/W_{\hL})^{F_q,\rm reg} \; \times_{\Hom (I_p^t,\hL^{\rm ab})} \Hom (\cG_q^t,\hL^{\rm ab}).$$
\end{Prop*}
\begin{proof} This follows from writing $\hL=\prod_{i} GL_{n_i}$ as a product of $GL_n$'s and applying \ref{prop_Born1} to each factor.
\end{proof}
\begin{Prop*} Let $\mathbf{\whG}=GL_n$. Fix a topological generator $M$ of $I_p^t$ and 
a Frobenius $\varphi$ in $\cG_q$. There is a map $h^{M,\varphi}_{\hL}$ depending on the pair $(M,\varphi)$
$$\Hom(\cG_q^t, \hL)/\hL \longleftarrow (\Hom(I_p^t, \hT)/W_{\hL})^{F_q,\rm cox}  \; \times_{\Hom (I_p^t,\hL^{\rm ab})} \Hom (\cG_q^t,\hL^{\rm ab}),$$
which is a section to $f_{\hL}$ in the sense that $f_{\hL}\circ h_{\hL}^{M,\varphi}={\rm id}$.
\end{Prop*}
\begin{proof} This follows from writing $\hL=\prod_{i} GL_{n_i}$ as a product of $GL_n$'s and applying \ref{prop_Born2} to each factor.
\end{proof}

\begin{Pt*}\label{Pt_Galois_par} Let $\hG=GL_n$. The injective map 
$h_{\hL}^{M,\varphi}$ of the last proposition may be written as 
$$  h_{\hL}^{M,\varphi}: X_\hL \lra \Hom(\cG_q^t, \hL)/\hL.$$
As will be recalled in \ref{Pt_rappel} below, the group $W(\hL)$ is a subgroup of the full permutation group on the set of blocks of the Levi subgroup $\hL$. Hence the map 
$h_{\hL}^{M,\varphi}$ is $W(\hL)$-equivariant, by its block wise definition. 
On the other hand, the map 
$$\Hom(\cG_q^t, \hL)/\hL \rightarrow \Hom(\cG_q^t, \hG)/\hG,$$ coming from $\hL\subset \hG$, is constant on $W(\hL)$-orbits and becomes injective on the quotient set. 
Composing $h_{\hL}^{M,\varphi}$ with the latter map gives an injective map 
$$ h_{[\hL]}^{M,\varphi}: X_{[\hL]}\longrightarrow  \Hom(\cG_q^t, \hG)/\hG.$$ 
\end{Pt*}

\section{Galois parametrization of Vinberg strata}
We keep all the notations and conventions introduced in the preceding section. 
\subsection{Dual Deligne-Lusztig pairs}
We recall some basic concepts from the representation theory of finite groups of Lie type, cf. \cite{C85} and \cite{DM91}.

\begin{Pt*}
To any $w\in W$, let $\hat{g}_w\in \hG$ such that $\hat{g}_w^{-1}F_q(\hat{g}_w)$ is a lift in $\hG$ of $w$ (Lang's theorem). Then put $\hT_w:=\hat{g}_w \hT \hat{g}_w^{-1}$.
Then $\hT_w$ is an $F_q$-stable maximal torus in $\hG$.
In this way, the conjugacy classes of $W$ are in bijection with the $\hG(\bbF_q)$-conjugacy classes of $F_q$-stable maximal tori in $\hG$. 
\vskip5pt
A {\it dual (Deligne-Lusztig) pair}  $(\hT',s)$ consists of 
an $F_q$-stable maximal torus $\hT'$ in $\hG$ and an element $s\in \hT'(\bbF_q)$.
Since $s$ is contained in a torus, it is a semisimple element.
The set of dual pairs admits a natural action of the finite group $\hG(\bbF_q)$.
There is a well-defined surjective map 

 \begin{eqnarray*}
\{\textrm{dual pairs $(\hT',s)$}\} / \hG(\bbF_q)& \lra & \{\textrm{semisimple elements in $\hG(\bbF_q)$}\}/ \hG(\bbF_q) \\
\textrm{$(\hT',s)$}  & \lmapsto & \textrm{$s$}.
\end{eqnarray*}
\end{Pt*}
Remark: Duality of algebraic and finite tori relates dual pairs to classical Deligne-Lusztig pairs, 
e.g. \cite[Prop. 13.13]{DM91}. We do not need this connection, so we will not make it precise here. 

\begin{Pt*}
Suppose that $(\hT',s)$ is a dual pair. Since $s\in \hG$ is a semisimple element, its centralizer $Z_{\hG}(s)$ is a reductive group. If the derived group $\hG^{\der}$ is simply connected, then $Z_{\hG}(s)$ is connected. We say that $\hT'$ (or the dual pair $(\hT',s)$) is {\it maximally split}, if there is an $F_q$-stable Borel subgroup of $Z_{\hG}(s)$ containing the maximal torus $\hT'$. 
All maximal split tori are $\hG(\bbF_q)$-conjugate in $\hG$ and, up to 
$\hG(\bbF_q)$-conjugacy, there is a unique maximally split torus in the connected centralizer $Z_{\hG}(s)^\circ$. Hence the natural map $(\hT',s) \mapsto s$ becomes injective upon restriction to maximally split pairs. One thus gets a bijection 
  \begin{eqnarray*}
\{\textrm{maximally split dual pairs $(\hT',s)$}\} / \hG(\bbF_q)& \simeq & \{\textrm{semisimple elements in $\hG(\bbF_q)$}\}/ \hG(\bbF_q) 
\end{eqnarray*}
\end{Pt*}
\begin{Pt*}
If the derived group $\hG^{\der}$ is simply connected, it is well-known that the inclusion map 
$\hG(\bbF_q)\ra \hG$ induces a bijection 
 \begin{eqnarray*}
\{\textrm{semisimple elements in $\hG(\bbF_q)$}\}/ \hG(\bbF_q) 
 & \simeq & \{\textrm{$F_q$-stable semisimple conjugacy classes in $\hG$}\}.
 \end{eqnarray*}
The inverse map is given by $c\mapsto c \cap \hG(\bbF_q)$, e.g. (proof of) \cite[Prop. 3.7.3]{C85}.
\end{Pt*}
\subsection{Jantzen parametrization}
Fix once for all a topological generator $(\zeta_{p^{i}-1})_{i=1}^\infty$ for the group $\varprojlim_i \bbF^{\times}_{p^i}$. We briefly recall Jantzen's parametrization of Deligne-Lusztig pairs from \cite[3.1]{J81}, in the language of dual pairs. 

\begin{Pt*}\label{Pt_Jantzen}
The group $W$ acts on $X_*(\hT)$ in the usual way and we may form the
semi-direct product 
$X_*(\hT)\rtimes W$. The group $X_*(\hT)\rtimes W$ acts on the set 
$W\times X_*(\hT)$ as follows:
$$ ^{(\nu,\sigma)}(w,\mu):=(\sigma w \sigma^{-1}, \sigma\mu+(q-\sigma w\sigma^{-1})\nu).$$

There is a well-defined map (depending on the choice of $(\zeta_{p^{i}-1})_{i=1}^\infty$) 
\begin{eqnarray*}
\textrm{$W\times X_*(\hT)$} & \longrightarrow & \{\textrm{dual pairs $(\hT',s)$}\} 
\end{eqnarray*}
given as follows.
Given $(w,\mu)\in W\times X_*(\hT)$,
we let $\hT':=\hT_{w^{-1}}$\footnote{The passage to $w^{-1}$ is with an eye towards Lem. \ref{Lem_cox} below.}. Let $t$ be such that $w^t=1$. Then the torus $\hT'$ splits over $\bbF_{q^t}$. The image of $(w,\mu)$ is then given by the dual pair  

$$(\hT_{w^{-1}} , \hat{g}_{w^{-1}} N_{(F_qw^{-1})^t / F_qw^{-1}}( \mu ) (\zeta_{q^{t}-1}) (\hat{g}_{w^{-1}})^{-1}) .$$

Here, $F_qw^{-1}$ is the endomorphism of $X_*(\hT)$ given by $F_qw^{-1}(\eta)=F_q\circ w^{-1}\circ \eta$ and for an arbitrary endomorphism 
$A$ of $X_*(\hT)$, we write $N_{A^t /A }:=\prod_{i=0}^{t-1} A^{i}$. In particular, given $\mu\in X_*(\hT)$, the cocharacter $N_{A^t /A }(\mu)$ sends an element $x\in\bbG_m$ to the element 
$\prod_{i=0}^{t-1} (A^{i}(\mu)(x)) $, the product being taken in $\hT$. 
The above map induces a well-defined bijection 
\begin{eqnarray*}
\{\textrm{$X_*(\hT)\rtimes W$-orbits in  $W\times X_*(\hT)$}\} & \iso & \{\textrm{dual pairs $(\hT',s)$}\}/ \hG(\bbF_q). 
\end{eqnarray*}
\end{Pt*}
\begin{Lem*}\label{Lem_cox}
In the notation of the preceding discussion we have for the element 

$$ s_{w,\mu}:= N_{(F_qw^{-1})^t / F_qw^{-1}}( \mu ) (\zeta_{q^{t}-1})\in \hT$$
that $$F_q.s_{w,\mu}= w.s_{w,\mu}.$$
\end{Lem*}
\begin{proof}
By construction $s_{w,\mu}$ is invariant under $F_qw^{-1}$. Indeed, it suffices to see that 
$(F_qw^{-1})^t$ fixes $\mu (\zeta_{q^{t}-1})$. But 
$$F^t_q(  \mu  (\zeta_{q^{t}-1}))=
\mu( (\zeta_{q^{t}-1})^{q^t})= \mu( \zeta_{q^{t}-1}),$$
and $F_q$ and $w^{-1}$ commute. So the claim follows from the fact that $w^{-t}=1$.
\end{proof}
\begin{Lem*} \label{lem_divisor}
Let $k\geq 1$. Then 
$$ N_{(F_qw^{-1})^{kt} / F_qw^{-1}}( \mu ) (\zeta_{q^{kt}-1})=
N_{(F_qw^{-1})^{t} / F_qw^{-1}}( \mu ) (\zeta_{q^{t}-1}).$$
\end{Lem*}
\begin{Proof}
Let
$A:=F_qw^{-1}$ and put $A_i:=A^{i}$ for $i=0,...,t-1$ and 
$A_j:=A^{jt}$ for $j=0,...,k-1$. Then
$$ \prod_j A_j =N_{(F^t_qw^{-t})^{k} / (F^t_qw^{-t})}=N_{(F^t_q)^{k} / (F_q)^t}=1+q^t+q^{2t}+\ldots +q^{(k-1)t}={\rm Norm}_{\bbF_{q^{kt}}/\bbF_{q^t}}.$$ But
${\rm Norm}_{\bbF_{q^{kt}}/\bbF_{q^t}}(\zeta_{q^{kt}-1})=\zeta_{q^t-1}$
by definition of $(\zeta_{p^{i}-1})_{i=1}^\infty$ and so
$$N_{(F_qw^{-1})^{kt} / F_qw^{-1}}( \mu ) (\zeta_{q^{kt}-1})= \prod_{i,j} A_i\circ A_j ( \mu ) (\zeta_{q^{kt}-1})=\prod_{i} A_i(\prod_j A_j ( \mu ) (\zeta_{q^{kt}-1}))=
(\prod_{i} A_i)( \mu ) (\zeta_{q^{t}-1})$$
and this is  
$N_{(F_qw^{-1})^{t} / F_qw^{-1}}( \mu ) (\zeta_{q^{t}-1})$.
\end{Proof}
 \begin{Rem*}\label{Pt_good} Let $\mathbf{\whG}=GL_n$. Let $(w,\mu)\in W\times X_*(\mathbf{\whT})$ with $\mu=(\mu_1,...,\mu_n)$. For any $1\leq i \leq n$, let $n_i$ be the smallest positive integer such that $w^{n_i}(i)=i$. According to a definition of Herzig \cite[Def. 6.19]{H09}, the pair $(w,\mu)$ is called {\it good} if for all $i$ 
$$ \sum_{k\; {\rm mod} \;n_i} \mu_{w^k(i)}q^k \not\equiv 0\mod \frac{q^{n_i}-1}{q^d-1} $$
for all $d\mid n_i, d\neq n_i$. The property {\it good} depends only on the 
$X_*(\mathbf{\whT})\rtimes W$-orbit of the pair $(w,\mu)$ and the restriction of the above bijection to good orbits yields a bijection \cite[Prop. 6.19]{H09}
\begin{eqnarray*}
\{\textrm{good $X_*(\mathbf{\whT})\rtimes W$-orbits in  $W\times X_*(\mathbf{\whT})$}\} & \iso & \{\textrm{maximally split dual pairs $(\hT',s)$}\}/ \hG(\bbF_q).
\end{eqnarray*}
\end{Rem*}
\begin{Ex*} \label{ex_det}
Let $\mathbf{\whG}=GL_n$ and assume $\hL=GL_{n_1}\times\cdot\cdot\cdot\times GL_{n_r}$. 
We identify the parabolic subgroup $W_\hL$ of $W=S_n$ with the product of symmetric groups $S_{n_1}\times\cdots\times S_{n_r}$.
Define
$$ \cox_{\hL}:=(1,...,n_1)\cdot\cdot\cdot (1,...,n_r)\in W_{\hL}.$$
For any $\mu\in X_*(\mathbf{\whT})$ Lem. \ref{Lem_cox} implies that 
$$s_{\cox_{\hL},\mu}\in \Hom(I_p^t, \hT)/W_{\hL})^{F_q,\rm cox}.$$
Writing 
 $\mu_i$ for the projection of $\mu$ to the $i$-th block and $\cox_i$ for the $i$-th cycle of $\cox_{\hL}$ we have in the $i$-th block 
 $$(s_{\cox_{\hL},\mu})_i=N_{(F_q\cox_i^{-1})^t / F_q\cox_i^{-1}}( \mu_i ) (\zeta_{q^{t}-1})=N_{(F_q\cox_i^{-1})^{n_i}/ F_q\cox_i^{-1}}( \mu_i ) (\zeta_{q^{n_i}-1}),$$
 where we have used \ref{lem_divisor} and the fact that $t$ may be taken to be the least common multiple of the $n_i$. Note that the top entry of the diagonal matrix 
 $(s_{\cox_{\hL},\mu})_i\in (\bbG_m)^{n_i}\subseteq GL_{n_i}$ is of the form $\zeta_{q^{n_i}-1}^{\sum_{k=1}^{n_i} \mu_{ik}q^k}$. Now according to \ref{Rem_irr}, the element $s_{\cox_{\hL},\mu}$ lies in the subset $\Hom(I_p^t, \hT)/W_{\hL})^{F_q,\rm reg}$ of $\Hom(I_p^t, \hT)/W_{\hL})^{F_q,\rm cox}$ 
if and only if for every fixed $i$, all elements in $(s_{\cox_{\hL},\mu})_i$ are pairwise different. 
Writing $\mu_{ik}$ for the $k$-th coordinate ($k=1,...,n_i)$ of $\mu_i$ this means that for every $i$ the element
$\zeta_{q^{n_i}-1}^{\sum_{k=1}^{n_i} \mu_{ik}q^k}$ is in no proper subfield of $\bbF_{q^{n_i}}$, i.e. that  $$\sum_{k=1}^{n_i} \mu_{ik}q^k \not\equiv 0\mod \frac{q^{n_i}-1}{q^d-1} $$
for all $d\mid n_i, d\neq n_i$. Summing up, we see that 
$s_{\cox_{\hL},\mu}\in \Hom(I_p^t, \hT)/W_{\hL})^{F_q,\rm reg}$ if and only if 
the pair $(\cox_{\hL},\mu)$ is good, in the language of \ref{Pt_good}.  

We finally note that 
 
 $$\det s_{\cox_{\hL},\mu}=\det \mu(\zeta_{q-1})$$ as elements in $\bbG_m$. In particular, for $\rho=(n-1,n-2,...,1,0)$ 
 $$\det s_{\cox_{\hL},\rho}=\zeta_{q-1}^{1+2+\ldots +n-1}.$$ Indeed, passing to the determinant in the $i$-th block and writing 
$ \mu_{ik}(\zeta_{q^{n_i}-1})$ for the $k$-th coordinate of $\mu_i(\zeta_{q^{n_i}-1})\in(\bbG_m)^{n_i}\subseteq GL_{n_i},$
 one obtains
 $$\det (s_{\cox_{\hL},\mu})_i= \prod_{k=1}^{n_i} \mu_{ik}(\zeta_{q^{n_i}-1})
 ^{1+q+\ldots +q^{n_i-1}}=\prod_{k=1}^{n_i} \mu_{ik}(\zeta_{q^{n_i}-1}
 ^{1+q+\ldots +q^{n_i-1}})=\prod_{k=1}^{n_i} \mu_{ik}(\zeta_{q-1})
 =\det\mu_i(\zeta_{q-1}),$$
 where we used $\zeta_{q^{n_i}-1}
 ^{1+q+\ldots +q^{n_i-1}}=\zeta_{q-1},$  by definition of $(\zeta_{p^{i}-1})_{i=1}^\infty$.
 Taking the product over all blocks $i$, one obtains $\det s_{\cox_{\hL},\mu}=\det \mu(\zeta_{q-1})$, as claimed.
  \end{Ex*}

\subsection{Dot actions on finite dual tori for $GL_n$} 
We assume $\mathbf{\whG}=GL_n$ throughout this subsection. 

\vskip5pt

As always in this case, let $\hT\subset\hB$ be the subgroups of diagonal and upper triangular matrices respectively. Let $\Delta\subset \Phi(\hG,\hT)$ be the standard subset of simple roots for this choice. 
We keep the choice of topological generator $(\zeta_{p^{i}-1})_{i=1}^\infty$ for the group $\varprojlim_i \bbF^{\times}_{p^i}$ from the preceding section. Let $\hL\in\widehat{\cL}$ be a standard Levi subgroup of $\hG$.
 \begin{Pt*} \label{Pt_rappel} We recall some stabilizers for the Weyl action on standard groups, cf.  \cite[2.6/2.8]{BZ77}. The group $W$ is identified with the subgroup of $N_\hG(\hT)$ consisting of elementary monomial matrices. 
 As such, it acts on the set of Levi subgroups of $\mathbf{\whG}$ by conjugation.
 The subset $W(\hL,\star)$ of $W$ consists of those elements that map $\hL$ to some standard Levi from $\widehat{\cL}$. Let $W_\hL \subset W$ be the parabolic subgroup of $W$ associated with $\hL$. The quotient set $W(\hL,\star)/W_\hL$ is in fact the permutation group on the set of blocks of $\hL$. It therefore identifies canonically with a subgroup of $W$. As such, its action on 
$\hT$ or $X_*(\hT)$ factors into an action on the quotient $\hT/W_\hL$ or $X_*(\hT)/W_\hL$.
On the other hand, the group
$W(\hL)=N_{\hG}(\hL)/\hL$ identifies with the subgroup of $W(\hL,\star)/W_\hL$ consisting of those permutations that preserve blocks of the same size. In other words, it is the stabilizer subgroup of $\hL$ in the group $W(\hL,\star)/W_\hL$. In particular $W(\hL)=1$ if and only if all blocks of $\hL$ have different seize. Finally, it is clear that the action of $W(\hL)$ on 
$X_*(\hT)$ commutes with the action of $W_\hL$.

\end{Pt*}

   
 \begin{Pt*}
  We first explain how to produce suitable systems of representatives in $X_*(\hT)$ for the 
 $W_\hL$-orbits in $\hT(\bbF_q)/W_\hL$.
 Supppose given a $\bbZ$-basis $\eta_1,...,\eta_n$ of $X_*(\hT)$ which is $W_{\hL}$-stable.  We have the surjection 
  $$ \ev_{\zeta_{q-1}}: X_*(\hT)\longrightarrow \hT(\bbF_q), \;\lambda\mapsto\lambda(\zeta_{q-1}).$$
  Via the basis $\eta_i$, it translates into the surjection $\bbZ^n\mapsto (\bbZ/(q-1))^n$.
  Hence the subset of $X_*(\hT)$ consisting of all linear combinations
  $ \sum_i a_i\eta_i$ with $0\leq a_i \leq q-2$ 
  bijects with  $\hT(\bbF_q)$ under $\ev_{\zeta_{q-1}}$. Obviously, $W_{\hL}$ still acts on the latter subset and we may choose a system of representatives $X_*(\hT)^{\hL}$ for the $W_{\hL}$-orbits. If we denote by $[\ev_{\zeta_{q-1}}]_{\hL}$ the map 
 $ \ev_{\zeta_{q-1}}$ followed by the quotient map $\hT(\bbF_q)\rightarrow \hT(\bbF_q)/W_\hL$, then by construction 
 $$ [\ev_{\zeta_{q-1}}]_{\hL}: X_*(\hT)^{\hL}\iso \hT(\bbF_q)/W_\hL.$$
 \end{Pt*}
 
 \begin{Rem*} \label{Rem_can_choice} 
 There are canonical choices for
 $X_*(\hT)^{\hL}$ (modulo our choice of $\hB$). For example the cocharacters $\eta_1,...,\eta_n$ given by $\eta_i(x)=\diag(1,...,x,1...,1)$, where $x$ is in position $i$, form a {\it canonical} basis of  $X_*(\hT)$, which is $W_{\hL}$-stable. Let us identify $X_*(\hT)$ with $\bbZ^n$ for the moment via this basis. 
 Let $I\subset \{1,...,n-1\}$ be the subset indexing the simple roots $\Delta(\hL)\subseteq \Delta$ belonging to $\hL$.
Imposing the dominance condition $0\leq a_i-a_{i+1}$ for all $i\in I$, one forms the subset
 $$ 'X_*(\hT)^{\hL}:=\{ (a_1,....,a_n)\in\bbZ^n: a_i\in \{0,...,q-2\} \;\forall i\; \hskip5pt \text{and} \hskip5pt a_{i+1}\leq a_i \; \forall \;i\in I\}.$$ It bijects with
 $\hT(\bbF_q)/W_\hL$ under the map $[\ev_{\zeta_{q-1}}]_{\hL}.$
 
 In the special case $\hL=\hT$, the condition of $W_{\hL}$-invariance is empty and there is a second canonical choice.  Namely, one may take the  $n$ "fundamental coweights"
 $\omega_i=\eta_1+\cdots +\eta_i$ as a basis of $X_*(\hT)$. 
 The first $n-1$ coordinates of a weight $\mu\in X_*(\hT)$ in this basis are given by the 
 values $\langle \mu,\hat{\alpha}_i \rangle$ where $\hat{\alpha}_i$ ranges over the $n-1$ simple roots of $\hT$, whence the resulting system of representatives is a subset of 
 $$\{ (a_1,....,a_n)\in\bbZ^n: 0\leq a_i-a_{i+1}\leq q-2\; \forall i\}. $$

 \end{Rem*}

  \begin{Pt*}
  We mark the dominant chamber (relative to $W_\hL$) with a chosen Coxeter element $\cox_{\hL}$ in $W_\hL$, for example
 $$ \cox_{\hL}=(1,...,n_1)\cdot\cdot\cdot (1,...,n_r)$$
 where $\hL=GL_{n_1}\times\cdot\cdot\cdot\times GL_{n_r}$. Note that we could have obtained any Coxeter element in $W_\hL$ by passing to a $W_\hL$-conjugate of the dominant chamber.
 \end{Pt*}

\begin{Pt*}
Let 
 $$\rho=(n-1,n-2,...,1,0)\in X_*(\hT).$$ There is the usual dot action
 $w\bullet \mu:= w(\mu+\rho)-\rho$ of $W$ on  $X_*(\hT)$. 
 \end{Pt*}
 \begin{Def*} \label{Def_generic} We say that a weight 
 $\mu\in X_*(\hT)^{\hL}$ is {\it generic} in $X_*(\hT)^{\hL}$ if its dot orbit $W(\hL)\bullet\mu$ is fully contained in $X_*(\hT)^{\hL}$.  
 \end{Def*}
 Of course, when $W(\hL)$ is trivial, the condition is empty. This is the case precisely if all blocks of $\hL$ have different seize, cf. \ref{Pt_rappel}. 
 In the following, we give an alternative characterization of the  property "generic".

\begin{Lem*}\label{lem_commute} The dot action of $W(\hL)$ on $X_*(\hT)$ commutes with the action of $W_\hL$. 
\end{Lem*}
\begin{proof} Suppose that $\hL=GL_{n_1}\times\cdot\cdot\cdot \times GL_{n_r}$.
Let $w\in W(\hL,\star)/W_\hL$ and $\mu\in X_*(\hT).$ 
Then $w\bullet \mu=w(\mu)+w(\rho)-\rho.$ The cocharacter $w(\rho)-\rho$ of $\hT$, viewed as the torus in the Levi subgroup $w\hL w^{-1}=GL_{n_{w(1)}}\times\cdot\cdot\cdot \times GL_{n_{w(r)}}$ is constant (with integer value) on each block of the latter Levi: its value in the block $GL_{n_{w(i+1)}}$ is the difference between the coordinate of $\rho$ at the position $n_1+\cdot\cdot\cdot+n_{i}+1$ and the coordinate of $\rho$ at the position $n_1+\cdot\cdot\cdot+n_{w(i)}+1$. Now if $w\in W(\hL)$, then 
$w\hL w^{-1}=\hL$ and the above observation implies that $w(\rho)-\rho$ is fixed by $W_\hL$. 
It follows for given $\nu\in W_\hL$ that 
$$ w\bullet (\nu(\mu))=w(\nu (\mu))+w(\rho)-\rho=w(\nu (\mu))+\nu (w(\rho)-\rho)= \nu(w(\mu)+w(\rho)-\rho)=\nu ( w\bullet \mu).$$

\end{proof}
\begin{Pt*} \label{Pt_modular_dot} There is a dot action of $W$ on the finite torus $\hT(\bbF_q)$ given by 
$$w\bullet t= w(t \rho(\zeta_{q-1}))\;\rho(\zeta_{q-1})^{-1}.$$ 
It makes the map 
$\ev_{\zeta_{q-1}}:  X_*(\hT)\ra \hT(\bbF_q)$
dot equivariant. Note also that $s\bullet t= s(t)\alpha_s^{-1}(\zeta_{q-1})$ for a simple reflection $s=s_\alpha$, because $s_\alpha(\rho)=\rho-\alpha_s$.

\end{Pt*}
\begin{Lem*} The induced dot action of the subgroup $W(\hL)$ on $\hT(\bbF_q)$ commutes with the action of $W_\hL$. 
In particular, the dot action of $W(\hL)$ passes to the quotient $\hT(\bbF_q)/W_\hL.$
\end{Lem*}
\begin{proof} Let $w\in W(\hL)$ and $\nu\in W_\hL$ and take $t\in\hT(\bbF_q)$.
Write $t=\lambda(\zeta_{q-1})$ for some $\lambda\in X_*(\hT)$.
The above lemma implies the equality
$w(\nu\bullet \lambda)=\nu\bullet w(\lambda)$. Evaluating on $\zeta_{q-1}$ and using that evaluation respects both the natural and the dot actions yields $w(\nu\bullet t)=\nu\bullet w(t)$.
\end{proof}
 
Using the isomorphism $$ [\ev_{\zeta_{q-1}}]_{\hL}: X_*(\hT)^{\hL}\iso \hT(\bbF_q)/W_\hL$$
from above, the dot action of $W(\hL)$ on $\hT(\bbF_q)/W_\hL$ induces by transport of structure an action of $W(\hL)$ on $X_*(\hT)^{\hL}$. We refer to it as the {\it modular dot action} of $W(\hL)$ on $X_*(\hT)^{\hL}$ and denote it by $\bullet_{\rm mod}$.

\begin{Lem*}\label{Lem_two_actions} Let $\lambda\in X_*(\hT)^{\hL}$. Then $\lambda$ is generic if and only if
the modular dot action of $W(\hL)$ on $\lambda$ coincides with the usual dot action of $W(\hL)$. 
\end{Lem*}
\begin{proof}
If the two actions coincide, then 
$$W(\hL)\bullet \lambda=W(\hL)\bullet_{\rm mod}\lambda \subset X_*(\hT)^{\hL},$$
and $\lambda$ is generic. Conversely, let $\lambda$ be generic and $w\in W(\hL)$. Let 
$\overline{t}=[\ev_{\zeta_{q-1}}]_{\hL}\lambda$.
Then both elements $w\bullet \lambda$ and $w\bullet_{\rm mod}\lambda$ of 
$X_*(\hT)^{\hL}$ map to $w\bullet \overline{t}$ under $[\ev_{\zeta_{q-1}}]_{\hL}$. Since the latter map is injective on $X_*(\hT)^{\hL}$, one has $w\bullet \lambda=w\bullet_{\rm mod}\lambda$.
 \end{proof}
 \begin{Pt*} \label{n-generic} For future purposes, we call the subset which corresponds to the generic weights under the bijection 
$X_*(\hT)^{\hL}\simeq \hT(\bbF_q)/W_\hL$, {\it the subset of generic orbits} in $\hT(\bbF_q)/W_\hL$, and denote it by $$(\hT(\bbF_q)/W_\hL)_{\rm gen}.$$ 
\end{Pt*}
\begin{Cor*} \label{Cor_dot_stable} The subset $(\hT(\bbF_q)/W_\hL)_{\rm gen}$ of  $\hT(\bbF_q)/W_\hL$ is stable under the dot action of $W(\hL)$. 
\end{Cor*}
\begin{Proof} Under the identification with the generic weights in $X_*(\hT)^{\hL}$, the dot action of $W(\hL)$ becomes what we called the modular dot action. By Lem. \ref{Lem_two_actions}, it therefore coincides with the usual dot action and the stability follows from the definition of generic weight, cf.  \ref{Def_generic}.
\end{Proof}

To show that our definition of generic weight is meaningful, we give an example of a canonical system of representatives $X_*(\hT)^{\hL}$, whose proportion of generic weights tends to $1$, as $p$ tends to infinity.

 \begin{Pt*} We reconsider the constructions appearing in Rem. \ref{Rem_can_choice}. So let $\eta_1,...,\eta_n$ given by $\eta_i(x)=\diag(1,...,x,1...,1)$, where $x$ is in position $i$, which form a {\it canonical} $W_{\hL}$-stable basis of  $X_*(\hT)$. Identify $X_*(\hT)$ with $\bbZ^n$ for the moment via this basis. 
Let $I\subset \{1,...,n-1\}$ be the subset indexing the simple roots $\Delta(\hL)$ belonging to $\hL$ and form the subset
 $$ 'X_*(\hT)^{\hL}:=\{ (a_1,....,a_n): a_i\in \{0,...,q-2\}\; \forall i\; \hskip5pt \text{and} \hskip5pt a_{i+1}\leq a_i \; \forall \;i\in I\}.$$
 It bijects with $\hT(\bbF_q)/W_\hL$ under the map $[\ev_{\zeta_{q-1}}]_{\hL}.$
In this particular situation, we say $(a_1,...,a_n)$ is {\it $n$-deep} in $'X_*(\hT)^{\hL}$, if 
$n-1\leq a_i \leq q-2 - (n-1)$ for all $i\in I$. Clearly, the proportion of $n$-deep weights in 
$'X_*(\hT)^{\hL}$ tends to $1$, as $p$ tends to infinity. 
\end{Pt*}
\begin{Lem*}
Any $n$-deep weight $\mu\in {'X}_*(\hT)^{\hL}$ is generic.
\end{Lem*}
\begin{proof} 
The set $'X_*(\hT)^{\hL}$ is stable under the natural action of $W(\hL)$, since this action preserves the block wise dominance condition $0\leq a_{i+1}-a_i$ for $i\in I$ (and trivially preserves the condition $a_i\in\{0,...,q-2\}$ for all $i$). The action stabilizes the subset of $n$-deep weights, since it preserves the block wise condition $n-1\leq a_i \leq q-2 - (n-1)$ for all $i\in I $. So if $w\in W(\hL)$, then 
$w(\mu)\in {'X}_*(\hT)^{\hL}$ is again $n$-deep. As we have seen in the proof of Lem. \ref{lem_commute} above, the weight $w(\rho)-\rho$ is constant on each block of $\hL$. Adding it to 
$w(\mu)$ therefore still preserves the dominance condition $0\leq a_{i+1}-a_i$ for $i\in I$. Moreover, the constant value on each block lies in the real interval $[-(n-1),n-1)]$. Hence adding 
$w(\rho)-\rho$ to the $n$-deep weight $w(\mu)$ still gives a weight in $'X_*(\hT)^{\hL}$, i.e.
$w\bullet \mu= w(\mu)+ w(\rho)-\rho\in {'X}_*(\hT)^{\hL}$. Hence $\mu$ is generic. 
\end{proof}


\subsection{Galois parametrization of Vinberg strata for $GL_n$}
We keep the assumptions from the previous subsection, in particular $\mathbf{\whG}=GL_n$.\begin{Lem*}\label{observation}
Suppose there is a map $N: X_*(\hT)\rightarrow \hT(\overline{\bbF}_p)$, which is equivariant with respect to the natural $W(\hL)$-action on $X_*(\hT)$ and $\hT(\overline{\bbF}_p)$. The map $\tilde{N}(\mu):=N(\mu+\rho)$ is then equivariant for the dot action of $W(\hL)$ on $X_*(\hT)$ and the natural action of $W(\hL)$ on $\hT(\overline{\bbF}_p)$.
\end{Lem*}
\begin{Proof} $N ((\nu\bullet \mu)+\rho)=N((\nu(\mu+\rho)-\rho)+\rho)=N(\nu(\mu+\rho))=
\nu N(\mu+\rho) $ for $\nu\in W(\hL)$.
\end{Proof}
\begin{Prop*} 
Let $\hL\in\widehat{\cL}$. The map 
$$ X_*(\hT)\longrightarrow \hT(\overline{\bbF}_p), \mu \mapsto s_{\cox_{\hL},\mu+\rho}$$
is equivariant for the dot action of $W(\hL)$ on $X_*(\hT)$ and the natural action of $W(\hL)$ on $\hT(\overline{\bbF}_p)$. 
\end{Prop*}
\begin{Proof}
The map $$ X_*(\hT)\rightarrow \hT(\overline{\bbF}_p), \mu \mapsto s_{w,\mu}.$$ 
is equivariant with respect to the natural $W(\hL)$-action on $X_*(\hT)$ and $\hT(\overline{\bbF}_p)$. Indeed, let $\nu\in W(\hL)$ and $w\in W_{\hL}$. Then 
$$ s_{w,\nu \mu}= N_{(F_qw^{-1})^t / F_qw^{-1}}( \nu \mu ) (\zeta_{p^{t}-1})= \nu 
( N_{(F_qw^{-1})^t / F_qw^{-1}}(\mu ) ) (\zeta_{p^{t}-1}) = \nu s_{w,\mu},$$
where we use that $\nu$ commutes with $w^{-1}\in W_{\hL}$ and $F_q$, hence also with $N_{(F_qw^{-1})^t / F_qw^{-1}}$. The proposition follows thus from Lem. \ref{observation}.
\end{Proof}

\begin{Cor*} 
Let $\hL\in\widehat{\cL}$. The map 
$ X_*(\hT)\longrightarrow \hT(\overline{\bbF}_p), \mu \mapsto s_{\cox_{\hL},\mu+\rho}$
induces a canonical map (i.e. independent of the choice of Coxeter element $\cox_{\hL}$ in $W_\hL$)
$$ \tau_\hL: X_*(\hT)\longrightarrow \Hom(I_p^t, \hT)/W_{\hL})^{F_q,\rm cox}, $$ 
which is
equivariant for the dot action of $W(\hL)$ on $X_*(\hT)$ and the natural action of $W(\hL)$ on the target. 
 \end{Cor*}
\begin{proof} The element $s_{\cox_{\hL},\mu+\rho}$ defines an element in 
$(\Hom(I_p^t, \hT)/W_{\hL})^{F_q}$. By the Jantzen parametrization, cf. \ref{Pt_Jantzen}, it only depends on the $W_{\hL}$-conjugacy class of $\cox_{\hL}$, i.e. is independent of the particular choice of Coxeter element. Moreover, since $F_q$ acts on $s_{\cox_{\hL},\mu+\rho}$ through $\cox_{\hL}$, cf. 7.2.2, it actually defines an element in $\Hom(I_p^t, \hT)/W_{\hL})^{F_q,\rm cox}$.
\end{proof}


\begin{Prop*} The map  $\tau_\hL$ factores through $\mu\mapsto \mu(\zeta_{q^t-1})$, i.e. into a map 
$$\hT(\bbF_{q^t})\rightarrow (\Hom(I_p^t, \hT)/W_{\hL})^{F_q}.$$
\end{Prop*}
\begin{Proof}
Let $\mu\in X_*(\hT)$.
One may rewrite  
 $$s_{\cox_{\hL},\mu} =\prod_{i=0,...,t-1} (F_q\cox_{\hL}^{-1})^{i}(\mu(\zeta_{q^t-1})),$$
 which shows that the value $\tau_\hL(\mu)=s_{\cox_{\hL},\mu+\rho}= s_{\cox_{\hL},\mu} s_{\cox_{\hL},\rho}$ depends only on $\mu(\zeta_{q^t-1})\in\hT$. \end{Proof} 

We have the dot action of $W(\hL)$ on $X_*(\hT)$ and the natural action of $W(\hL)$ on $\hL^{\ab}$. Endow the product $X_*(\hT)\times \hL^{\ab}$ with the diagonal action. The following proposition follows then by construction. 
\begin{Prop*} 
Let $\hL\in\widehat{\cL}$. Fix an 
an arithmetic Frobenius $\varphi\in \cG_q$. 
The map 
$$ \tau_\hL\times{\id}: X_*(\hT)\times \hL^{\ab} \longrightarrow
\Hom(I_p^t, \hT)/W_{\hL})^{F_q,\rm cox}\times \hL^{\ab} = X_{\hL} $$
is
$W(\hL)$- equivariant. The induced map $X_*(\hT)\times \hL^{\ab}\rightarrow X_{[\hL]}$ is constant on $W(\hL)$-orbits.

\end{Prop*}
\begin{Pt*}
We may compose the map from of the proposition with the injection 
$h_{[\hL]}^{M,\varphi}: X_{[\hL]}\hookrightarrow \Hom(\cG_q^t, \hG)/\hG$ from \ref{Pt_Galois_par}.
\end{Pt*}
\begin{Cor*} 
Let $\hL\in\widehat{\cL}$. Fix an 
an arithmetic Frobenius $\varphi\in \cG_q$. 
The map
$$ X_*(\hT)\times \hL^{\ab} \longrightarrow \Hom(\cG_q^t, \hG)/\hG, \;(\mu,z)\mapsto \rho_{\mu,z}$$
is
constant on $W(\hL)$-orbits in the source.
The $n$-dimensional Galois representation $\rho_{\mu,z}$ 
is semisimple if and only the $W_{\hL}$-orbit of $\tau_{\hL}(\mu)$ is regular, i.e. has maximal size equal to the seize of $W_\hL$. 
\end{Cor*}
\begin{proof} The last assertion on semisimplicity follows from \ref{Prop_irr}, the rest is clear. 
\end{proof}
\begin{Pt*}\label{Pt_explicit}
Let us unravel the various definitions involved in the map of the corollary.
Suppose that $\hL= GL_{n_1}\times\cdot\cdot\cdot \times GL_{n_r}$ and 
$\cox_{\hL}=(1,...,n_1)\cdot\cdot\cdot (1,...,n_r)$.
The map associates to a pair $(\mu,z)$ with $\mu\in X_*(\hT)$ and 
$z=(z_{1},...,z_{r})\in \hL^{\ab}\simeq (\bbG_m)^r$ (via $\det$ on each block), the
isomorphism class of the tame Galois representation $\rho_{\mu,z}$, where
 $M$ acts through
$s_{\cox_\hL,\mu+\rho}$ and $\varphi$ acts through the block matrix
$(s_{1},...,s_{r})\in \hL$ given by

$$s_i=\left( \begin{array}{rrrrr}
0 & 1 &  & &\\
 & 0 & 1 & &\\
 &  & \ddots & \ddots & \\
 &  &  &0 &1\\
(-1)^{n_i+1}z_i&  &  & &0\\
\end{array}\right),$$
whenever $n_i\geq 2$.
Note that indeed $\det s_i= (-1)^{n_i+1} ( (-1)^{n_i+1}z_i) = z_i$ in this case.
If $n_i=1$, then $\varphi$ acts of course through the scalar endomorphism given by $z_i$.
\end{Pt*} 
\begin{Pt*}
Now fix a system of representatives $X_*(\hT)^{\hL}$ in $X_*(\hT)$ for the $W_{\hL}$-orbits in $\hT(\bbF_q)$, as in the preceding section. The generic weights in $X_*(\hT)^{\hL}$, are then in dot equivariant bijection with the generic orbits $(\hT(\bbF_q)/W_\hL)_{\rm gen}$, cf. \ref{n-generic} and \ref{Cor_dot_stable}.
Restricting the weight component of the above map in the above corollary to the generic weights in $X_*(\hT)^{\hL}$ gives the following second corollary. 
\end{Pt*}
\begin{Cor*}\label{cor_main_appli} Let $\hL\in\widehat{\cL}$ and choose a Frobenius $\varphi\in\cG_q$. The induced map
$$(\hT(\bbF_q)/W_\hL)_{\rm gen} \times \hL^{\ab} \longrightarrow X_{[\hL]} $$ 
is
constant on $W(\hL)$-orbits.
\end{Cor*} 
Again, one may compose the preceding map with the injection 
$h_{[\hL]}^{M,\varphi}: X_{[\hL]}\hookrightarrow \Hom(\cG_q^t, \hG)/\hG$ from \ref{Pt_Galois_par}.
to arrive at a map 
$$(\hT(\bbF_q)/W_\hL)_{\rm gen} \times \hL^{\ab} \longrightarrow \Hom(\cG_q^t, \hG)/\hG, \hskip10pt (\bar{t},z)\mapsto \rho_{\bar{t},z},$$
which is
constant on $W(\hL)$-orbits in the source. 
It is compatible with central characters and determinants (up to a cyclotomic shift) in the following sense. For $z=(z_{1},...,z_{r})\in \hL^{\ab}\simeq (\bbG_m)^r$ write 
$\det z:=z_1\cdots z_r$.

\begin{Prop*}\label{prop_main_appli} Let $(\bar{t},z)\in (\hT(\bbF_q)/W_\hL)_{\rm gen} \times \hL^{\ab}$. Then 
$$\det\rho_{\bar{t},z}(M)= (\det \bar{t})\cdot \zeta_{q-1}^{1+2+\ldots+ n-1} \hskip10pt \text{and}\hskip10pt \det\rho_{\bar{t},z}(\varphi) = \det z.$$
\end{Prop*}
\begin{Proof}
 We may suppose that $\hL=GL_{n_1}\times\cdot\cdot\cdot\times GL_{n_r}$ and that
$$ \cox_{\hL}=(1,...,n_1)\cdot\cdot\cdot (1,...,n_r)\in W_{\hL}.$$
Suppose that $\bar{t}$ corresponds to the generic weight $\mu\in X_*(\hT)^{\hL}$, i.e. $\mu(\zeta_{q-1})=t \mod W_{\hL}$.
Then $\rho_{\bar{t},z}$ maps $M$ to
$s_{\cox_\hL,\mu+\rho}$ and $\varphi$ to the block matrix
$(s_{1},...,s_{r})\in \hL$, as described in \ref{Pt_explicit}. According to example \ref{ex_det},
$$\det s_{\cox_\hL,\mu+\rho}= \det s_{\cox_\hL,\mu}\cdot\det s_{\cox_\hL,\rho}=
\det \mu(\zeta_{q-1}) \cdot \det \rho(\zeta_{q-1}) = (\det t) \cdot \zeta_{q-1}^{1+2+\ldots+ n-1}.$$
Moreover, $\det (s_{1},...,s_{r})=\prod_i \det s_i = \prod_i z_i= \det z.$   
\end{Proof}
\begin{Rem*}\label{Rem_indep} By construction, and in addition to the choice of $M$ and $\varphi$, the map 
from Cor. \ref{cor_main_appli} depends a priori on the choice of the system of representatives 
$X_*(\hT)^{\hL}$ for the orbits in $\hT(\bbF_q)/W_{\hL}$. In the light of the preceding proposition, one may ask, if the map is in fact independent of that latter choice, i.e. depends only on the set $\hT(\bbF_q)/W_{\hL}$. We do not know the answer to this at the moment (but see \ref{calcul_independent} in the case of $GL_2$). 
\end{Rem*}
\begin{Pt*}\label{Pt_Galois_par}
Recall from \ref{Pt_stratification} the stratum 
$$S_{\hL}=(\hT(\bbF_q)\times W.\hL^{\ab})/W=\hT(\bbF_q)/W_\hL \times \hL^{\ab}$$
corresponding to $\hL$ in the stratification  of the Satake scheme $S=V^{(1)}_{\hT\subset\hG,0,\bbF_q}/W$.
It has its diagonal $W(\hL)$-action, where $W(\hL)$ acts via the dot action on the first factor. We denote by $$S_{\hL,{\rm gen}}:=(\hT(\bbF_q)/W_\hL)_{\rm gen} \times \hL^{\ab}$$
its {\it generic} part, i.e. the part indexed by the generic orbits. It is $W(\hL)$-stable. The map from Cor. \ref{cor_main_appli} can then be viewed as a $W(\hL)$-invariant 
morphism of $\bbF_q$-schemes
$$ S_{\hL,{\rm gen}}\longrightarrow X_{[\hL]}.$$ 
Taking into account the inclusion $X_{[\hL]}(\overline{\bbF}_p)\hookrightarrow \Hom(\cG_q^t, \hG(\overline{\bbF}_p))/\hG(\overline{\bbF}_p)$, this can be viewed as a Galois parametrization of the generic part of the stratum $S_{\hL}$. 
\end{Pt*}
\subsection{The case of $GL_2$}\label{section_GL2}
Assume that $\mathbf{\whG}=GL_2$.
The goal of this section is to show that the $W(\hL)$-invariant morphism $S_{\hL,{\rm gen}}\rightarrow X_{[\hL]}$ appearing in \ref{cor_main_appli} equals the restriction to the stratum of type $\hL$ of the $\bbF_q$-morphism $L$ of \cite[sec. 6]{PS25} in the case $F:=\bbQ_q$.  In particular, the map on quotients yields an isomorphism $$X_{\hL}/W(\hL)\simeq X^{\rm reg}_{[\hL]}$$ in this case.  
Recall from \ref{Pt_Galois_par} the inclusion $X^{\rm reg}_{[\hL]}(\overline{\bbF}_p)\hookrightarrow \Hom(\cG_q^t, \hG(\overline{\bbF}_p))/\hG(\overline{\bbF}_p)$. According to the case $n=2$ in \ref{Prop_irr}, this gives bijections $$X^{\rm reg}_{[\hG]}(\overline{\bbF}_p) \iso \{ \text{irreducible $\cG_q$-representations of dimension $2$}\}/ \simeq $$
and 
$$\hskip50pt X^{\rm reg}_{[\hT]}(\overline{\bbF}_p)\iso \{ \text{reducible semisimple $\cG_q$-representations of dimension $2$}\}/\simeq. $$
However, in the $GL_2$-case, much more is true: as shown in \cite{PS25}, the stratified space $X^{\rm reg}_{[\hT]} \cup X^{\rm reg}_{[\hG]}$ has, after fixing the determinant, the structure of a projective curve, namely a finite chain of projective lines with transversal intersection points corresponding, essentially, to the points in 
$X^{\rm reg}_{[\hG]}$.
\begin{Pt*} Recall the canonical isomorphism $I_p^t\iso \varprojlim_i \bbF^{\times}_{p^{i}},$
cf. \ref{Pt_can_iso}. Let for all $i$
$$ \omega_{i}: I_p^t\lra \bbF^{\times}_{p^{i}}$$
denote the projection to the $i$-th component in the above isomorphism. 
Since $n=2$, we shall only need the two characters 
$\omega_{2f}, \omega_f$ in the following. Here, $q=p^f$ with some $f\geq 1$. As before, we fix an element $\varphi\in \cG_q$ lifting the Frobenius $x\mapsto x^q$. We may assume that $\omega_f(\varphi)=1$. 
\end{Pt*}
\begin{Pt*}Recall \cite[2.5]{PS25} that for any $h\geq 0$, there is a $\cG_q$-representation 
 $\ind (\omega_{2f}^h)$,
 which is uniquely determined by the two conditions 
 $$\det  \ind (\omega_{2f}^h) =  \omega^h_f\hskip10pt \text{and} \hskip10pt\ind (\omega_{2f}^h) |_{I_p} = \omega_{2f}^{h} \oplus \omega_{2f}^{qh}.$$ 
Denote for $z\in \overline{\bbF}_p^\times$ by $\unr(z)$ the unramified character of $\cG_q$ sending $\varphi^{-1}$ to $z$. 
\end{Pt*}
\begin{Pt*} We identify $X_*(\hT)=\bbZ^2$ via the canonical basis $\eta_1,\eta_2$, cf. \ref{Rem_can_choice}. Let
 $$X_*(\hT)_{q-\reg}=\{ (a_1,a_2)\in\bbZ^2: 0\leq a_1-a_2<q-1\}.$$
The group $X_*(\hT)^0:=\bbZ(1,1)$ acts on this set and the quotient of $X_*(\hT)_{q-\reg}$ modulo the action of 
$(q-1)X_*(\hT)^0$ bijects with $\hT(\bbF_q)$ upon evaluation at $\zeta_{q-1}$. We will identify the two sets from now on.
Let $\hat{\alpha}=(1,-1)$ and $i=0,...,\frac{q-1}{2}$. 
Then $$i\hat{\alpha}=(i,-i)\in X_*(\hT)_{q-\reg} \hskip5pt \text{and}\hskip5pt
(q-1-i,i)\in X_*(\hT)_{q-\reg}$$
for $i\neq 0, \frac{q-1}{2}$. 
\end{Pt*}
\begin{Pt*}
Let $\hL\in\cL$. We then have the map 
$$(\hT(\bbF_q)/W_\hL)_{\rm gen} \times \hL^{\ab} \longrightarrow \Hom(\cG_q^t, \hG)/\hG$$
from \ref{cor_main_appli}. 
According to \ref{prop_main_appli}, we may fix the determinants $$\det \overline{t}=\zeta_{q-1}^n \hskip5pt \text{and}\hskip5pt \det z$$ with $(\zeta_{q-1}^n,\det z)\in \bbG_m(\bbF_q)\times\bbG_m$.
There are only two cases to consider, $\hL=\mathbf{\whG}$ and $\hL=\hT$. 
\end{Pt*}

Let $\hL=\mathbf{\whG}$. Then $W_{\hL}=W$ and $W(\hL)=1$ and the condition "generic" is empty. Also $\hL^{\ab}=\bbG_m$ via the determinant and $\det z$ is just $z$ (there is only one block).
We will see that the map 
$$\hT(\bbF_q)/W \times \bbG_m \longrightarrow \Hom(\cG_q^t, \hG)/\hG$$
is in fact independent of the choice of system of representatives for the $W$-orbits in $\hT(\bbF_q)$, cf. as discussed in Rem. \ref{Rem_indep}.
There are two cases to consider, $n$ even and $n$ odd.

\begin{Pt*}
Let $n$ be even. Write $w$ for the generator of $W$. We have the sequence of $W$-orbits in $\mathbf{\whT}(\bbF_q)$
$$
t_i \cdot \diag(\zeta^s,\zeta^s), \hskip5pt t_i^w \cdot \diag(\zeta^s,\zeta^s)
$$ 
where $n=2s$ and $t_i:=\diag(\zeta^i,\zeta^{-i})$ (and $t_i^w$ its $w$-conjugate) for $i=0,...,\frac{q-1}{2}$.  
\vskip5pt

Recall that $i\hat{\alpha}\in X_*(\hT)_{q-\reg}$ for $i<\frac{q-1}{2}$ and clearly $\ev_{\zeta_{q-1}} (i\hat{\alpha}) =t_i$ and $$\ev_{\zeta_{q-1}}(i\hat{\alpha}+(s,s))=t_i\cdot \diag(\zeta^s,\zeta^s).$$

One has 
$$ i\hat{\alpha}+\rho= i(\eta_1-\eta_2)+\eta_1=2i\eta_1+\eta_1-i\eta_1-i\eta_2=(2i+1)\eta_1- i (\eta_1+\eta_2).$$
Now $$\;\;\;\;N_{(F_qw)^2 / F_qw}((2i+1)\eta_1)(\zeta_{q^{2}-1})=
\left(\begin{array}{cc}
\zeta_{q^{2}-1}^{2i+1} & 0\\
0 & \zeta_{q^{2}-1}^{(2i+1)q}
\end{array} \right)$$
$$N_{(F_qw)^2 / F_qw}( i (\eta_1+\eta_2))(\zeta_{q^{2}-1})=
\left(\begin{array}{cc}
\zeta_{q-1}^{i} & 0\\
0 & \zeta_{q-1}^{i}
\end{array} \right),$$
using $\zeta_{q^2-1}^{1+q}=\zeta_{q-1}$ in the second equation. All in all, 
$$s_{w,i\hat{\alpha} +\rho}=N_{(F_qw)^2 / F_qw}( i\hat{\alpha}+\rho) (\zeta_{q^{2}-1})=
\left(\begin{array}{cc}
\zeta_{q^{2}-1}^{2i+1} & 0\\
0 & \zeta_{q^{2}-1}^{(2i+1)q}
\end{array} \right)\cdot \left(\begin{array}{cc}
\zeta_{q-1}^{-i} & 0\\
0 & \zeta_{q-1}^{-i}
\end{array} \right).$$
More generally, $$s_{w,(i\hat{\alpha}+(s,s)) +\rho}=\left(\begin{array}{cc}
\zeta_{q^{2}-1}^{2i+1} & 0\\
0 & \zeta_{q^{2}-1}^{(2i+1)q}
\end{array} \right)\cdot \left(\begin{array}{cc}
\zeta_{q-1}^{-(i+s)} & 0\\
0 & \zeta_{q-1}^{-(i+s)}
\end{array} \right).$$ 
It follows that 
$$\rho_{t_i\cdot \diag(\zeta^s,\zeta^s),z^2}=\ind(\omega_{2f}^{2i+1})\otimes\omega_{f}^{-(i+s)}\otimes \unr(z), $$
as $\cG_q$-representations, cf. \cite[2.1]{PS25}\footnote{In \cite{PS25}, the parameter $z$ is called $z_2$.}. For example, if $z=1$ and $n=s=0$ (the so-called basic even case of \cite[4.5.2]{PS25}) this amounts to 

$$\rho_{t_i,1}=\ind(\omega_{2f}^{r+1})\otimes\omega_{f}^{s(r)}$$
with $r:=2i$ and $s(r):=-\frac{r}{2}$. This is the irreducible $2$-dimensional Galois representation attached to the origin of the $r$-th irreducible component $\mathcal{C}^r$ of the curve $X_{d_0}|_{\pr_2=1}$, in the notation of \cite[4.5.2]{PS25}.
\end{Pt*}
\begin{Pt*} \label{calcul_independent} 
In this paragraph, and with regard towards Rem. \ref{Rem_indep} above, we verify concretely in the above case that the isomorphism class of the Galois representation is independent of the chosen representative of the $W$-orbit of $t_i$. Recall that $(q-1-i,i) \in X_*(\hT)_{q-\reg}$ for $i>0$ and clearly $\ev_{\zeta_{q-1}} ((q-1-i,i )) =t^w_i$. One has 
$$ (q-1-i,i)+\rho=(q-i)\eta_1+i\eta_2=((q-2i)\eta_1+i(\eta_1+\eta_2)$$ and this is equal to
$$ (q-2i)\eta_1+(2i-1-(i-1))(\eta_1+\eta_2)=(q-r)\eta_1+r-1+s(r-2),
$$
where $r:=2i$ and $s(r)=-\frac{r}{2}$ as above. 
We obtain from that in the same way as above, 
$$\rho_{t^w_i,1}=\ind(\omega_{2f}^{q-r})\otimes\omega_{f}^{s(r-2)+r-1}.$$
It remains to remark that there is indeed an isomorphism of Galois representations 
$$\ind(\omega_{2f}^{r+1})\otimes\omega_{f}^{s(r)} \simeq \ind(\omega_{2f}^{q-r})\otimes\omega_{f}^{s(r-2)+r-1},$$
e.g. \cite[4.5.2]{PS25}.

\end{Pt*}
\begin{Pt*}
We turn to the case of $n$ odd. We have the sequence of $W$-orbits in $\mathbf{\whT}(\bbF_q)$
$$
t_i \cdot \diag(\zeta^s,\zeta^s), \hskip5pt t_i^w \cdot \diag(\zeta^s,\zeta^s)
$$ 
where $n=2s-1$ and $t_i:=\diag(\zeta^{i-1+\frac{q-1}{2}},\zeta^{-i+\frac{q-1}{2}})$ (and $t_i^w$ is its $w$-conjugate) for $i=1,...,\frac{q-1}{2}$. 
One has $\ev_{\zeta_{q-1}}(i\hat{\alpha}-\rho+(s,s))=
t_i\cdot \diag(\zeta^s,\zeta^s)$ and 
$$ (i\hat{\alpha}-\rho)+\rho= i(\eta_1-\eta_2)=2i\eta_1-i\eta_1-i\eta_2=(2i)\eta_1- i (\eta_1+\eta_2).$$
In a complete analogous manner to the above, we obtain as $\cG_q$-representations $$\rho_{t_i\cdot \diag(\zeta^s,\zeta^s),z^2}=\ind(\omega_{2f}^{2i})\otimes\omega_{f}^{-(i+s)}\otimes \unr(z).$$
For example, if $z=1$ and $n=q-2$ (the so-called basic odd case of \cite[4.5.5]{PS25}) this amounts to 

$$\rho_{t_i,1}=\ind(\omega_{2f}^{r+1})\otimes\omega_{f}^{s(r)}$$
with $r:=2i-1$ and $s(r):=-\frac{r+1}{2}$. This is the irreducible $2$-dimensional Galois representation attached to the origin of the $r$-th irreducible component $\mathcal{C}^r$ of the curve $X_{d_{q-2}}|_{\pr_2=1}$, in the notation of \cite[4.5.5]{PS25}. Again, one may verify that $\rho_{t^w_i,1}$ is isomorphic to $\rho_{t_i,1}.$
\end{Pt*}
 
\vskip5pt 
Let $\hL=\hT$. Then $W_{\hL}=1$ and $W(\hL)=W$. Also $\hL^{\ab}=\hT$. As indicated above, we identify 
$X_*(\hT)_{q-\reg}\mod (q-1)X_*(\hT)^0$ with  $\hT(\bbF_q)$. The map becomes 
$$\hT(\bbF_q)_{\rm gen} \times \hT \longrightarrow \Hom(\cG_q^t, \hG)/\hG$$
and is $W$-invariant for the diagonal $W$-action on the source (and the dot action on the first factor $\hT(\bbF_q)$). We will see that this holds in fact on the full set $\hT(\bbF_q)$.
Again, there are two cases to consider, $n$ even and $n$ odd.

\begin{Pt*}
Let $n$ be even. We have the sequence of elements in $\mathbf{\whT}(\bbF_q)$
$$
t_i \cdot \diag(\zeta^s,\zeta^s), \hskip5pt t_i^w \cdot \diag(\zeta^s,\zeta^s)
$$ 
where $n=2s$ and $t_i:=\diag(\zeta^i,\zeta^{-i})$ (and $t_i^w$ its $w$-conjugate) for $i=0,...,\frac{q-1}{2}$.  
\end{Pt*}
\begin{Pt*}

We first treat the characters $t_i$. Recall that $i\hat{\alpha}\in X_*(\hT)_{q-\reg}$ for $i<\frac{q-1}{2}$ and $\ev_{\zeta_{q-1}} (i\hat{\alpha}) =t_i$ and $$\ev_{\zeta_{q-1}}(i\hat{\alpha}+(s,s))=t_i\cdot \diag(\zeta^s,\zeta^s).$$
One has 
$$ i\hat{\alpha}+\rho= i(\eta_1-\eta_2)+\eta_1=2i\eta_1+\eta_1-i\eta_1-i\eta_2=(2i+1)\eta_1- i (\eta_1+\eta_2).$$
Now $$\;\;\;\;N_{(F_q1)^1 / F_q1}((2i+1)\eta_1)(\zeta_{q-1})=
\left(\begin{array}{cc}
\zeta_{q-1}^{2i+1} & 0\\
0 & 1
\end{array} \right)$$
$$N_{(F_q1)^1 / F_q1}( i (\eta_1+\eta_2))(\zeta_{q-1})=
\left(\begin{array}{cc}
\zeta_{q-1}^{i} & 0\\
0 & \zeta_{q-1}^{i}
\end{array} \right).$$
All in all, 
$$s_{1,i\hat{\alpha} +\rho}=N_{(F_q1)^1 / F_q1}( i\hat{\alpha}+\rho) (\zeta_{q-1})=
\left(\begin{array}{cc}
\zeta_{q-1}^{2i+1} & 0\\
0 &1
\end{array} \right)\cdot \left(\begin{array}{cc}
\zeta_{q-1}^{-i} & 0\\
0 & \zeta_{q-1}^{-i}
\end{array} \right).$$
More generally, $$s_{1,(i\hat{\alpha}+(s,s)) +\rho}=\left(\begin{array}{cc}
\zeta_{q-1}^{2i+1} & 0\\
0 & 1
\end{array} \right)\cdot \left(\begin{array}{cc}
\zeta_{q-1}^{-(i+s)} & 0\\
0 & \zeta_{q-1}^{-(i+s)}
\end{array} \right).$$ 
It follows with $z=(z_1,z_2)\in \hT$ that as tame $\cG_q$-representations
$$\rho_{t_i\cdot \diag(\zeta^s,\zeta^s),z}
=\left(\begin{array}{cc}
\unr(z_1)\omega_{f}^{2i+1} & 0\\
0 & \unr(z_2)
\end{array} \right)
\otimes\omega_{f}^{-(i+s)}
$$
in the language of the appendix or \cite[2.1]{PS25}. For example, if $z=1$ and $n=s=0$ (the so-called basic even case of \cite[4.5.2]{PS25}) this amounts to 

$$\rho_{t_i,1}=\left(\begin{array}{cc}
\unr(x)\omega_{f}^{r+1} & 0\\
0 & \unr(x^{-1})
\end{array} \right)
\otimes\omega_{f}^{s(r)}
$$
with $r:=2i$ and $s(r):=-\frac{r}{2}$. This is the semisimple reducible $2$-dimensional Galois representation attached to a point $[x:1]\in\bbP^1\setminus\{0,\infty\}=\cC^r\setminus\{0,\infty\}$ on the $r$-th irreducible component $\mathcal{C}^r$ of the curve $X_{d_0}|_{\pr_2=1}$, in the notation of \cite[4.5.2]{PS25}. Here, $x$ is an affine coordinate near the origin $0\in \mathcal{C}^r$.
\end{Pt*}
\begin{Pt*} We now treat the characters $t^w_i$. Recall that $(q-1-i,i)\in X_*(\hT)_{q-\reg}$ for $i>0$ and $\ev_{\zeta_{q-1}} ((q-1-i,i)) =t^w_i$ and $$\ev_{\zeta_{q-1}}((q-1-i,i)+(s,s))=t^w_i\cdot \diag(\zeta^s,\zeta^s).$$ To verify the $W$-invariance in the next step, it seems practical to do the case $i+1$ here. Going through the calculation  \ref{calcul_independent} with $i+1$ 
shows $$ (q-1-(i+1),i+1)+\rho=(q-r-2)\eta_1+r+1+s(r),$$
with $r=2i$ and $s(r)=\frac{r}{2}$ as usual. For example, in the basic even case $z=1$ and $n=s=0$ it follows that 
$$\rho_{t^w_{i+1},1}
=\left(\begin{array}{cc}
\unr(y)\omega_{f}^{q-r-2} & 0\\
0 & \unr(y^{-1})
\end{array} \right)
\otimes\omega_{f}^{s(r)+r+1}
$$
This is the semisimple reducible $2$-dimensional Galois representation attached to a point $[1:y]\in\bbP^1\setminus\{0,\infty\}=\cC^r\setminus\{0,\infty\}$ on the $r$-th irreducible component $\mathcal{C}^r$ of the curve $X_{d_0}|_{\pr_2=1}$, in the notation of \cite[4.5.2]{PS25}. Here, $y$ is an affine coordinate near $\infty\in \mathcal{C}^r$.

\end{Pt*}
\begin{Pt*} We check the $W$-invariance of the map 
$$\hT(\bbF_q) \times \hT \longrightarrow \Hom(\cG_q^t, \hG)/\hG$$
with respect to the diagonal action of $W$ on the source $\hT(\bbF_q)\times \hT$.
 According to 
\ref{Pt_modular_dot} the modular dot action of $W$ on the finite torus $\hT(\bbF_q)$ is given by 
$w\bullet t= w(t)\hat{\alpha}^{-1}(\zeta)$. In particular,
$$ w\bullet t_i=w\bullet  \diag(\zeta^{i},\zeta^{-i})=\diag(\zeta^{-i},\zeta^{i})(\zeta^{-1},\zeta)=t^w_{i+1}.$$ More generally, 
$$ w\bullet (t_i \cdot \diag(\zeta^s,\zeta^s))=t^w_{i+1}\cdot \diag(\zeta^s,\zeta^s).$$

Taking into account the $W$-action on the second factor $\hT$ 
(which permutes the two coordinates $(z_1,z_2)$ on $\hT$), the $W$-invariance amounts to having isomorphisms of Galois representations
$$  \rho_{t_{i}\cdot \diag(\zeta^s,\zeta^s), z}\simeq \rho_{t^w_{i+1}\cdot \diag(\zeta^s,\zeta^s), z^w}.$$
By our above computation of both sides, this reduces to having 
isomorphisms of Galois representations
$$
\left(\begin{array}{cc}
\unr(x)\omega_{f}^{r+1} & 0\\
0 & \unr(x^{-1})
\end{array} \right)
\otimes\omega_{f}^{s(r)}
\simeq
\left(\begin{array}{cc}
\unr(y)\omega_{f}^{q-2-r} & 0\\
0 & \unr(y^{-1})
\end{array} \right)
\otimes\omega_{f}^{s(r)+r+1}
$$
relative to the identification $x=\frac{1}{y}$. But this is well-known, e.g. \cite[4.5.2]{PS25}.
\end{Pt*}
\begin{Pt*}
The case of odd $n$ is completely analogous, we therefore omit the details. \end{Pt*}

\section{Application to pro-$p$-Iwahori Hecke algebras}
We keep all the notation. Additionally, let $\mathbf{T}\subset\mathbf{G}$ be the connected reductive groups dual to $\mathbf{\whT}\subset \mathbf{\whG}$. We identify the Weyl group of 
$(\mathbf{G}, \mathbf{T})$ with $W$. Our reference for the (pro-$p$-) Iwahori Hecke algebra of a $p$-adic reductive group is Vignéras work \cite{V16}. 
\subsection{The pro-$p$-Iwahori Hecke algebra}
Let $F/\bbQ_p$ be a finite extension with residue field $\bbF_q=o_F/m_F$. 
Let $\mathbf{N}$ be the normalizer of $\mathbf{T}$ in $\mathbf{G}$ and 
put $W^{(1)}:=\mathbf{N}(F)/ \mathbf{T}(1+m_F)$ and $\widetilde{W}:=\mathbf{N}(F)/ \mathbf{T}(o_F)$. These groups fit into the short exact sequence 
$$ 1 \lra \bbT \lra W^{(1)} \lra \widetilde{W}\lra 1$$
where $\bbT=\mathbf{T}(\bbF_q)$. For any subset $S\subset \widetilde{W}$, denote by $S^{(1)}$ its preimage in $W^{(1)}$. Let $S_{\rm aff}\subset \widetilde{W}$ be a subset of simple affine roots. For any $s\in S_{\rm aff}$, denote by $h_s\in X_*(\mathbf{T})$ the corresponding cocharacter and write $\bbT_s:=h_s(\bbF_q^\times) \subset \bbT$ for the corresponding finite torus. If $s\in S_{\rm aff}^{(1)}$, write $\bbT_s$ for the finite torus associated with the image of $s$. 
Let $\ell$ be the length function on $\widetilde{W}$ 
pulled back to $W^{(1)}$. Let $\bfq$ be an indeterminate. 

\begin{Def*} \label{defgenericprop}
The \emph{pro-$p$-Iwahori Hecke algebra} is the $\bbZ[\bfq]$-algebra $\cH^{(1)}(\bfq)$ defined by generators
$$
\cH^{(1)}(\bfq):=\bigoplus_{w\in W^{(1)}} \bbZ[\bfq] T_w
$$
and relations:
\begin{itemize}
\item braid relations
$$
T_wT_{w'}=T_{ww'}\quad\textrm{for $w,w'\in W^{(1)}$ if $\ell(w)+\ell(w')=\ell(ww')$}
$$
\item quadratic relations
$$
T_s^2=\bfq T_{s^2} + c_sT_s\quad \textrm{if $s\in S_{\aff}^{(1)}$},
$$
where 
$$
c_s:=\sum_{t\in \bbT_s}T_t.
$$
\end{itemize}

The $\bbZ[\bfq]$-basis $(T_w)_{w\in W^{(1)}}$ is called the \emph{Iwahori-Matsumoto basis}.
\end{Def*}

\subsection{The Bernstein basis}

\begin{Lem*}
Let $w_1,w_2 \in W^{(1)}$. Then:
$$
\ell(w_1)+\ell(w_2)-\ell(w_1w_2)\in 2\bbN
$$
$$
\frac{\ell(w_1)+\ell(w_2)-\ell(w_1w_2)}{2}\leq \ell(w_1)
$$
$$
\frac{\ell(w_1)+\ell(w_2)-\ell(w_1w_2)}{2}\leq \ell(w_2).
$$
\end{Lem*}

\begin{proof}
These facts are particular cases of \cite[\S 4.4]{V16}. They may also be proved directly using classical properties of the length function.
\end{proof}

\begin{Pt*}\label{qdc}
Consistently with \emph{loc. cit}, we set
$$
\bfq_{w_1,w_2}:=\bfq^{\frac{\ell(w_1)+\ell(w_2)-\ell(w_1w_2)}{2}}\quad\in\bfq^{\bbN}.
$$
Furthermore, we set
$$
\bfd_{w_1,w_2}:=\bfq^{\ell(w_1)-\frac{\ell(w_1)+\ell(w_2)-\ell(w_1w_2)}{2}}\quad\in\bfq^{\bbN}
$$
and
$$
\bfc_{w_1,w_2}:=\bfq^{\ell(w_2)-\frac{\ell(w_1)+\ell(w_2)-\ell(w_1w_2)}{2}}\quad\in\bfq^{\bbN}.
$$
Set also $\bfq_w:=\bfq^{\ell(w)}$. Then we have the following identities in $\bfq^{\bbN}$:
\begin{eqnarray*}
\bfq_{w_1}\bfq_{w_2}&=&\bfq_{w_1w_2}\bfq_{w_1,w_2}^2\\
\bfq_{w_1}&=&\bfq_{w_1,w_2}\bfd_{w_1,w_2}\\
\bfq_{w_2}&=&\bfq_{w_1,w_2}\bfc_{w_1,w_2}.
\end{eqnarray*}
\end{Pt*}

\begin{Th*}\textbf{\emph{ (Vignéras)}}
\label{BbasisV16}There exists a basis $(E(w))_{w\in W^{(1)}}$ of $\cH^{(1)}(\bfq)$ over $\bbZ[\bfq]$
(associated with the antidominant orientation) 
$$
\cH^{(1)}(\bfq)=\bigoplus_{w\in W^{(1)}}\bbZ[\bfq]E(w)
$$
satisfying the product formula
$$
E(e^{\nu})E(w)=\bfq_{\nu,w}E(e^{\nu}w)\quad\textrm{for $\nu\in X_*(\mathbf{T})^{(1)}$ and $w\in W^{(1)}$}.
$$

It is called the \emph{Bernstein basis}.
\end{Th*}

\begin{proof}
This is a particular case of \cite[5.25-5.26]{V16}. More precisely, Vignéras associates an \emph{alcove walk basis} $(E_o(w))_{w\in W^{(1)}}$ to any \emph{orientation} $o$, such that $E_o(w)=T_w$ if $w\in\Omega^{(1)}$ and satisfying the product formula. 
\end{proof}

For $\nu\in X_*(\mathbf{T})^{(1)}$, we also write simply $E(\nu)$ for $E(e^{\nu})$. In particular, the algebra $\cH^{(1)}(\bfq)$ contains the following \emph{subalgebra}, which moreover is \emph{commutative}:

\begin{Def*}
We call the sub-$\bbZ[\bfq]$-algebra
$$
\cA^{(1)}(\bfq):=\bigoplus_{\nu\in X_*(\mathbf{T})^{(1)}}\bbZ[\bfq]E(\nu)\subset \cH^{(1)}(\bfq)
$$
the \emph{Vignéras subalgebra}.
\end{Def*}

\subsection{The Iwahori Hecke algebra}
Here is the analogue for the group $\widetilde{W}$ instead of $W^{(1)}$.

\begin{Def*} \label{defgenericIwahori}
The \emph{Iwahori Hecke algebra} is the $\bbZ[\bfq]$-algebra $\cH(\bfq)$ defined by generators
$$
\cH(\bfq):=\bigoplus_{w\in \widetilde{W}} \bbZ[\bfq] T_w
$$
and relations:
\begin{itemize}
\item braid relations
$$
T_wT_{w'}=T_{ww'}\quad\textrm{for $w,w'\in \widetilde{W}$ if $\ell(w)+\ell(w')=\ell(ww')$}
$$
\item quadratic relations
$$
T_s^2=\bfq + (\bfq-1)T_s\quad \textrm{if $s\in S_{\aff}$}.
$$
\end{itemize}

The $\bbZ[\bfq]$-basis $(T_w)_{w\in \widetilde{W}}$ is called the \emph{Iwahori-Matsumoto basis}.
\end{Def*}

Here is the analogue of \ref{BbasisV16} 
 and for the Iwahori Hecke algebra.

\begin{Th*} \textbf{\emph{(Vignéras)}} \label{BpresHq}
The Iwahori Hecke algebra admits a \emph{Bernstein basis} 
$$
\cH(\bfq)=\bigoplus_{w\in \widetilde{W}}\bbZ[\bfq]E(w)
$$
having the following property: 
For $(\nu,w)\in X_*(\mathbf{T})\times \widetilde{W}$, there is the product formula
$$
E(e^{\nu})E(w)=\bfq_{\nu,w}E(e^{\nu}w)\ ;
$$
in particular
$$
\cA(\bfq):=\bigoplus_{\nu \in X_*(\mathbf{T})}\bbZ[\bfq] E(\nu)
$$
is a commutative sub-$\bbZ[\bfq]$-algebra.
\end{Th*}
 
\begin{Rem*} \label{rem_split} The exact sequence $$1\rightarrow \bfT(\bbF_q)\rightarrow X^{(1)}_*(\bfT)\rightarrow X_*(\bfT)\rightarrow 1$$ 
is split: the lattice $X_*(\bfT)$ identifies canonically with the free group $\bfT(F)/\bfT(o_F)$ and one may embed the latter, via a choice of uniformizer in $o_F$, into $\bfT(F)$. We choose once for all a splitting and thus regard $X_*(\bfT)$ as a subgroup of $X^{(1)}_*(\bfT)$. Note that the above product formula implies 
$E(t\nu)=E(t)E(\nu)$ for $t\in \bfT(\bbF_q), \nu\in X_*(\bfT)$, since $\bfq_{t,\nu}=1$ in this case.
As a consequence, the multiplication map 
$$\bbZ[\bfT(\bbF_q)]\otimes_{\bbZ}\cA(\bfq)\iso \cA^{(1)}(\bfq), \hskip5pt E(t)\otimes E(\nu)\mapsto E(t)E(\nu)$$
is a ring isomorphism.
\end{Rem*}

\subsection{Generic Bernstein isomorphism}

\begin{Pt*}
Recall the Vignéras sub-$\bbZ[\bfq]$-algebra 
$
\cA(\bfq)$
of $\cH(\bfq)$, cf. \ref{BpresHq}. On the dual side, recall the `toral part' $V_{\mathbf{\whT}\subset\mathbf{\whG},\rho_{\ad}}$ of the Zhu monoid $V_{\mathbf{\whG},\rho_{\ad}}$, which maps to $\bbA_1=\Spec(\bbZ[\bfq])$ by $d\rho_{\ad}$, cf. \ref{VTGrhoad}. 
\end{Pt*}

\begin{Th*} \label{ThBernsteingeneric}
Over the identity $X_*(\mathbf{T})=X^*(\mathbf{\whT})$, the $\bbZ[\bfq]$-linear map 
\begin{eqnarray*}
\sB(\bfq):\cA(\bfq) & \lra & \bbZ[V_{\mathbf{\whT}\subset\mathbf{\whG},\rho_{\ad}}]\\
E(\nu) & \lmapsto & (e_1^{\nu}\otimes e_2^{-\nu_-})|_{\rho_{\ad}}
\end{eqnarray*}
is an isomorphism of $\bbZ[\bfq]$-algebras.
\end{Th*}

\begin{proof}
The product formula in $\cA(\bfq)$ is $E(\nu)E(\nu')=\bfq_{\nu,\nu'}E(\nu+\nu')$ with
$$
\bfq_{\nu,\nu'}=\bfq^{\frac{\ell(\nu)+\ell(\nu')-\ell(\nu+\nu')}{2}}.
$$
The result follows thus directly from \ref{Cor_basis}.
\end{proof}
This generalizes \cite[Thm. 6.1.1]{PS23}, which was the case $\mathbf{\whG}=GL_2$.

\begin{Prop*} \label{Bernsteinequiv}
The Bernstein isomorphism $\sB(\bfq)$ is $W$-equivariant. In particular, it induces an isomorphism
$$
\xymatrix{
\sB(\bfq)^{W}:\cA(\bfq)^{W}\ar[r]^<<<<<{\sim} & \bbZ[V_{\mathbf{\whT}\subset\mathbf{\whG},\rho_{\ad}}]^{W}.
}
$$
\end{Prop*}

\begin{proof}
Indeed, by definition of the action of $W$ of $V_{\mathbf{\whT}\subset\mathbf{\whG}}$ \ref{defWfaction}, and because $\nu_-\in X^*(\mathbf{\whT})_-$ depends only on the $W$-orbit of $\nu\in X^*(\mathbf{\whT})$, we have
$$
\forall w\in W,\quad w(e_1^{\nu}\otimes e_2^{-\nu_-})=e_1^{w(\nu)}\otimes e_2^{-\nu_-}=e_1^{w(\nu)}\otimes e_2^{-w(\nu)_-}.
$$
\end{proof}

\subsection{Bernstein isomorphism at $\bfq=0$} \label{VGT0}

At $\bfq=0$, Bernstein's isomorphism \ref{ThBernsteingeneric} reads:

\begin{Cor*} 
Over the identity $X_*(\mathbf{T})=X^*(\mathbf{\whT})$, the $\bbZ[\bfq]$-linear map 
\begin{eqnarray*}
\sB(0):\cA(0) & \lra & \bbZ[V_{\mathbf{\whT}\subset\mathbf{\whG},0}]\\
E(\nu) & \lmapsto & (e_1^{\nu}\otimes e_2^{-\nu_-})|_{0}
\end{eqnarray*}
is an isomorphism of $\bbZ$-algebras. It is $W$-equivariant.
\end{Cor*}

Accordingly, and as we already did in \ref{Cor_basis}, we set 
$$
\forall\nu\in X^*(\mathbf{\whT}),\quad E(\nu):=(e_1^{\nu}\otimes e_2^{-\nu_-})|_{0}\in \bbZ[V_{\mathbf{\whT}\subset\mathbf{\whG},0}].
$$
Recall that the quantity $\bfq_{\nu,\nu'}=\bfq^{\frac{\ell(\nu)+\ell(\nu')-\ell(\nu+\nu')}{2}}$ is equal to $1$ if there exists a Weyl chamber $\fC$ such that $\nu$ and $\nu'$ belong to the closure $\overline{\fC}$, and it is divisible by $\bfq$ otherwise. 
Thus
$$
\bbZ[V_{\mathbf{\whT}\subset\mathbf{\whG},0}]=\bigoplus_{\nu\in X^*(\mathbf{\whT})}\bbZ E(\nu)
$$
with product formula 
$$
\forall \nu,\nu'\in X^*(\mathbf{\whT}), \quad E(\nu)E(\nu')=
\left\{ \begin{array}{ll}
E(\nu+\nu') & \textrm{if $\exists\fC: \nu,\nu' \in X^*(\mathbf{\whT})\cap \overline{\fC}$}\\ 
0 & \textrm{otherwise}.
\end{array} \right.
$$

\subsection{Augmented Bernstein isomorphism at $\bfq=q=0\in \bbF_q$}

\begin{Pt*}

Our fixed generator $\zeta_{q-1}$ of $\bbF_q^\times$, 
yields an isomorphism $\bbZ/(q-1)\bbZ\simeq \bbF_q^\times$. This induces a ring isomorphism 
between the ring of functions $\bbF_q[(\bbZ/(q-1)\bbZ)^r]$ of $\hT(\bbF_q)$ and the group ring 
$\bbF_q[\mathbf{T}(\bbF_q)]$ of the finite torus $\mathbf{T}(\bbF_q)$. This corresponds to an isomorphism between the constant schemes $\hT(\bbF_q)$ and $\bbT^\vee:=\Spec \bbF_q[\mathbf{T}(\bbF_q)]$.
\end{Pt*}
Recall the augmented semigroup $V^{(1)}_{\mathbf{\whT}\subset\mathbf{\whG},0,\bbF_q}$ over $\bbF_q$, cf. \ref{subsec_augmented_semi}.

\begin{Prop*} \label{ThBernstein1generic}\label{Bernstein1equiv}
 Over the isomorphism  $\bbF_q[\bbT]\simeq \cO(\hT(\bbF_q))$, the Bernstein isomorphism $\sB(0)_{\bbF_q}$ induces 
an isomorphism of $\bbF_q$-algebras
$$ 
\sB^{(1)}(0)_{\bbF_q}:\cA^{(1)}(0)_{\bbF_q} \iso \bbF_q[V^{(1)}_{\mathbf{\whT}\subset\mathbf{\whG},0,\bbF_q}].
$$
The isomorphism is $W$-equivariant and induces an isomorphism
$$
\xymatrix{
\sB^{(1)}(0)_{\bbF_q}^{W}:\cA^{(1)}(0)_{\bbF_q}^{W}\ar[r]^<<<<<{\sim} & \bbF_q[V^{(1)}_{\mathbf{\whT}\subset\mathbf{\whG},0,\bbF_q}]^{W}.
}
$$
\end{Prop*}
\begin{proof} This is clear since 
$\cA^{(1)}(0)_{\bbF_q}=\bbF_q[\bbT]\otimes \cA^{(1)}(0)$
and $\bbF_q[V^{(1)}_{\mathbf{\whT}\subset\mathbf{\whG},0,\bbF_q}]= \cO(\hT(\bbF_q))\otimes \bbZ[V_{\mathbf{\whT}\subset\mathbf{\whG},0}]$
as $\bbF_q$-algebras and these tensor product decompositions respect the $W$-action.

\end{proof}
\subsection{Galois parametrization of central Hecke strata for $GL_n$}
Let $\bfG=GL_n$ and $G=\bfG(F)$. 
\begin{Pt*} The 
pro-$p$-Iwahori Hecke algebra $\cH^{(1)}_{\bbF_q}$ of the reductive $p$-adic group $G$ (with coefficients in $\bbF_q$)  is the 
convolution algebra over $\bbF_q$ on the set of 
double cosets of $G$ relative to the choice of a pro-$p$-Iwahori subgroup in $G$, \cite[2.1]{V16}.
It satisfies 
 $$\cH^{(1)}_{\bbF_q}=\cH^{(1)}(\bfq) \otimes_{\bbZ[\bfq], \bfq\mapsto 0\in\bbF_q}\bbF_q,$$
 cf. \cite[Ex. 4.10]{V16}. 
Moreover, the theory of the Bernstein basis, developed in loc.cit., shows that 
the maximal commutative subring  $\cA^{(1)}_{\bbF_q}\subset \cH^{(1)}_{\bbF_q}$ relative to the antidominant orientation is isomorphic, in a $W$-equivariant way, to $\cA^{(1)}(0)_{\bbF_q}$, cf. \cite[Cor. 5.28]{V16}. It follows that the center $Z(\cH^{(1)}_{\bbF_q})$ of $\cH^{(1)}_{\bbF_q}$ equals the ring of invariants
$\cA^{(1)}(0)_{\bbF_q}^{W}$. 
\end{Pt*}
\begin{Pt*} 
The Bernstein isomorphism, cf. 
\ref{Bernstein1equiv}, thus induces an isomorphism 
$$ \Xi:=\Spec Z(\cH^{(1)}_{\bbF_q}) \iso V^{(1)}_{\mathbf{\whT}\subset\mathbf{\whG},0,\bbF_q}/W=S.$$
The target has its stratification
$$ S=\bigcup_{\hL\in\widehat{\cL}} S_{\hL},$$
cf. \ref{Pt_stratification}, indexed by the standard Levi subgroups $\hL\in\widehat{\cL}$ in the dual group $\hG$. It
induces a corresponding stratification of $\Xi$. Let $\mathbf{L}$ be the standard Levi of $\mathbf{G}$ corresponding\footnote{The standard Levi subgroups of $\mathbf{G}$ and $\mathbf{\whG}$ are in canonical bijection via the duality.}  to $\hL$. The corresponding group $W(\mathbf{L})$ identifies canonically with $W(\hL)$.
Let $\Xi_{\mathbf{L}}$ (resp. $\Xi_{\mathbf{L},\rm gen}$) be the inverse image of (the generic part $S_{\hL,\rm gen}$ of) the stratum $S_{\hL}$, cf. \ref{Pt_Galois_par}. It has its induced $W(\mathbf{L})$-action (which equals the dot action on the finite torus part).
Pulling back the Galois parametrization from
\ref{Pt_Galois_par} to $\Xi_{\mathbf{L},\rm gen}$ yields a 
Galois parametrization 
$$ \Xi_{\mathbf{L},\rm gen}(\overline{\bbF}_p) \longrightarrow \Hom(\cG_q^t, \hG(\overline{\bbF}_p))/\hG(\overline{\bbF}_p)$$
which is $W(\mathbf{L})$-invariant. 
\end{Pt*}
\begin{Pt*}
In the case of $\bfG=GL_2$ and $F=\bbQ_q$, the unramified extension of $\bbQ_p$ with residue field $\bbF_q$, our discussion in 
\ref{section_GL2} shows that the two maps 
$$\Xi_{\mathbf{L},\rm gen}(\overline{\bbF}_p) \rightarrow \Hom(\cG_q^t, \hG(\overline{\bbF}_p))/\hG(\overline{\bbF}_p),$$
for $\hL=\hG$ and $\hL=\hT$, are defined on the whole of $\Xi_{\mathbf{L}}(\overline{\bbF}_p)$ and that the set-theoretic map on $\Xi(\overline{\bbF}_p)=\Xi_{\mathbf{G}}(\overline{\bbF}_p)\coprod \Xi_{\mathbf{T}}(\overline{\bbF}_p)$ into the space of Galois representations, is induced from the scheme morphism $$\sL: \Xi\rightarrow X(q)_{\bbF_q}$$ between $\Xi$ and the {\it curve of semisimple $2$-dimensional Galois representations} $X(q)_{\bbF_q}$ constructed in \cite{PS25}. In particular the two maps yield
bijections 
 $$\Xi_{\mathbf{G}}(\overline{\bbF}_p) \iso \{ \text{irreducible $\cG_q$-representations of dimension $2$}\}/ \simeq $$
and 
$$\hskip35pt (\Xi_{\mathbf{T}}/W)(\overline{\bbF}_p) \iso \{ \text{reducible semisimple $\cG_q$-representations of dimension $2$}\}/\simeq.$$
\end{Pt*}

Remark: Analyzing the finer structure of the Galois parametrization of 
$\Xi_{\mathbf{L},\rm gen}$ for varying $\hL$ in the case of $\bfG=GL_3$ or $GL_4$ is a problem, which we hope to address in future work. 

\section{Appendix: The standard scheme}

Let $\mathbf{\whG}$ be a connected split group over $\bbZ$, with maximal torus and Borel subgroup
$\mathbf{\whT}\subset\mathbf{\whB}\subset \mathbf{\whG}$. Let $W=W(\mathbf{\whG},\mathbf{\whT})$ be the finite Weyl group associated to $(\mathbf{\whG},\mathbf{\whT})$.

\begin{Def}\label{defstandardscheme}
The $\bbZ$-scheme
$$
\mathds{U}_{\mathbf{\whG}}:=\Spec(\bbZ[X^*(\mathbf{\whT})^-])
$$
is called the \emph{standard scheme} associated to $(\mathbf{\whG},\mathbf{\whT},\mathbf{\whB})$.
\end{Def}

\subsection*{The orbit structure}

\begin{Pt}
The scheme $\mathds{U}_{\mathbf{\whG}}$ is a toric scheme over $\bbZ$ for the torus $\mathbf{\whT}$. Thus it is a normal integral scheme containing $\mathbf{\whT}=\Spec(\bbZ[X^*(\mathbf{\whT})])$ as an open subscheme, and which is equipped with an action of $\mathbf{\whT}$ extending the action on itself by multiplication. Moreover it is affine, so it is a diagonalizable $\bbZ$-monoid scheme with monoid of characters $X^*(\mathds{U}_{\mathbf{\whG}})=X^*(\mathbf{\whT})^-$ and group of units $\mathbf{\whT}$, and the $\mathbf{\whT}$-action on $\mathds{U}_{\mathbf{\whG}}$ is the restriction to $\mathbf{\whT}\times \mathds{U}_{\mathbf{\whG}}$ of the monoid law. We will give a precise description of the orbit structure of the $\mathbf{\whT}$-scheme $\mathds{U}_{\mathbf{\whG}}$, in particular of the \emph{boundary}
$$
\partial(\mathbf{\whT}\subset \mathds{U}_{\mathbf{\whG}}):=\mathds{U}_{\mathbf{\whG}}\setminus \mathbf{\whT}
$$
of $\mathbf{\whT}$ in the \emph{partial compactification} $\mathds{U}_{\mathbf{\whG}}$.
\end{Pt}

\begin{Pt}
Let $\mathbf{T}\subset \mathbf{G}$ be the corresponding dual reductive groups. The Borel subgroup $\mathbf{\whB}$ determines a Borel
$\mathbf{B}$ in $\mathbf{G}$ containing $\mathbf{T}$.
We denote by
$$
\xymatrix{
\lan\ ,\ \ran:X^*(\mathbf{T}) \times X^*(\mathbf{\whT}) \ar[r] & \bbZ
}
$$
the perfect duality pairing coming from the identity $X^*(\mathbf{\whT}) =X_*(\mathbf{T})$. If $\cS\subset X^*(\mathbf{T})$ is a subset, we denote by $\cS^{\perp}\subset X^*(\mathbf{\whT})$ the orthogonal subgroup. In particular, one may consider the subgroup $\Delta^{\perp}$ where $\Delta\subset X^*(\mathbf{T})$ is the basis of simple roots defined by $\mathbf{B}$. 
\end{Pt}

\begin{Def}
The $\bbZ$-scheme
$$
\mathds{Z}_{\mathbf{\whG}}:=\Spec(\bbZ[\Delta^{\perp}\cap X^*(\mathbf{\whT})^-])
$$
is called the \emph{zero scheme} associated to $(\mathbf{\whG},\mathbf{\whT},\mathbf{\whB})$.
\end{Def}

\begin{Pt}
The scheme $\mathds{Z}_{\mathbf{\whG}}$ is an affine toric scheme over $\bbZ$ for the torus 
$$
\mathds{S}_{\mathbf{\whG}}:=\Spec(\bbZ[\Delta^{\perp}]).
$$
In particular, it is a diagonalizable monoid scheme with monoid of characters $X^*(\mathds{Z}_{\mathbf{\whG}})=\Delta^{\perp}\cap X^*(\mathbf{\whT})^-$ and group of units $\mathds{S}_{\mathbf{\whG}}$. 
\end{Pt}

\begin{Pt}
A parabolic subgroup of $\mathbf{\whG}$ containing $\mathbf{\whB}$ is called a \emph{standard parabolic}, and the Levi factor containing $\mathbf{\whT}$ of a standard parabolic is called a \emph{standard Levi}. We denote by $\widehat{\cL}$ be the set of standard Levi's of $\mathbf{\whG}$. It is in 1-1 correspondence with the power set of the set $\widehat{\Delta}$ of simple roots defined by $\mathbf{\whB}\supset \mathbf{\whT}$:
\begin{eqnarray*}
 \cP(\widehat{\Delta}) & \iso & \widehat{\cL}  \\
\whI & \lmapsto & \textrm{$\mathbf{\whL}(\whI)$ such that $\widehat{\Delta}_{\mathbf{\whL}(\whI)}=\whI$}.
\end{eqnarray*}
A similar correspondence holds for the group $\mathbf{G}$ with respect to the maximal torus $\mathbf{T}$ and Borel $\mathbf{B}$:
\begin{eqnarray*}
 \cP(\Delta) & \iso & \cL  \\
I & \lmapsto & \textrm{$\mathbf{L}(I)$ such that $\Delta_{\mathbf{L}(I)}=I$}.
\end{eqnarray*}
The two correspondences are related by the canonical bijection $\widehat{\Delta}\xrightarrow{\sim}\Delta$, $\widehat{\alpha}\mapsto\alpha$: the latter induces a canonical bijection $\cP(\widehat{\Delta})\xrightarrow{\sim}\cP(\Delta)$, $\whI \mapsto  I$, hence a canonical bijection $\widehat{\cL}\xrightarrow{\sim}\cL$, 
$\mathbf{\whL}\mapsto \mathbf{L}$. 
\end{Pt}

\begin{Pt}
For $\mathbf{\whL}\in\widehat{\cL}$, we have the standard $\mathbf{\whT}$-toric scheme
$$
\mathds{U}_{\mathbf{\whL}}:=\Spec(\bbZ[X^*(\mathbf{\whT})^{-/\mathbf{\whL}}])
$$
associated to $(\mathbf{\whL},\mathbf{\whB}\cap\mathbf{\whL},\mathbf{\whT})$ and the zero $\dsS_{\mathbf{\whL}}$-toric scheme
$$
\mathds{Z}_{\mathbf{\whL}}:=\Spec(\bbZ[\Delta_\mathbf{L}^{\perp}\cap X^*(\mathbf{\whT})^{-/\mathbf{\whL}}])
$$
associated to $(\mathbf{\whL},\mathbf{\whB}\cap\mathbf{\whL},\mathbf{\whT})$. Moreover, we have the zero $\dsS_{\mathbf{\whL}}$-toric scheme
$$
\mathds{Z}_{\mathbf{\whL}\subset \mathbf{\whG}}:=\Spec(\bbZ[\Delta_\mathbf{L}^{\perp}\cap X^*(\mathbf{\whT})^-])
$$

associated to $(\mathbf{\whL}\subset \mathbf{\whG},\mathbf{\whB},\mathbf{\whT})$. 
Let us note that the torus $\dsS_{\mathbf{\whL}}=\Spec(\bbZ[\Delta_\mathbf{L}^{\perp}])$ has dimension
$$
\dim \dsS_{\mathbf{\whL}}=\rank\mathbf{\whT}-|\Delta_{\mathbf{\whL}}|=\dim \dsU_{\mathbf{\whL}}-|\Delta_{\mathbf{\whL}}|=\dim \dsU_{\mathbf{\whG}}-|\Delta_{\mathbf{\whL}}|.
$$
\end{Pt}

\begin{Prop} 
\begin{enumerate}
\item Let $\mathbf{\whL}\in\widehat{\cL}$.
\begin{enumerate}
\item The homomorphisms of $\bbZ$-monoid schemes 
$$
\xymatrix{
\mathbf{\whT}=\dsU_{\mathbf{\whT}} \ar[r] & \dsU_{\mathbf{\whL}}\ar[r] & \dsU_{\mathbf{\whG}}
}
$$
induced by the inclusions $X^*(\mathbf{\whT})^- \subset X^*(\mathbf{\whT})^{-/\mathbf{\whL}}\subset X^*(\mathbf{\whT})$ are open immersions.

 \item The additive map
\begin{eqnarray*}
\bbZ[X^*(\mathbf{\whT})^-] & \lra & \bbZ[\Delta_L^{\perp}\cap X^*(\mathbf{\whT})^-] \\
e^{\nu} & \lmapsto & \left\{ \begin{array}{ll}
e^{\nu} & \textrm{if $\nu\in\Delta_L^{\perp}$}\\ 
0 & \textrm{otherwise}
\end{array} \right.
\end{eqnarray*}
is a surjective homomorphism of rings, which is compatible with comultiplications, i.e. which corresponds to a closed immersion of 
$\bbZ$-semigroup schemes
$$
\xymatrix{
\dsZ_{\mathbf{\whL}\subset \mathbf{\whG}}\ar[r] & \dsU_{\mathbf{\whG}}.
}
$$
We denote by $e_{\mathbf{\whL}}\in  \dsU_{\mathbf{\whG}}$ the image of the unit element of the monoid $\dsZ_{\mathbf{\whL}\subset \mathbf{\whG}}$.

\item We have
$
 \dsS_{\mathbf{\whL}}=\dsZ_{\mathbf{\whL}\subset \mathbf{\whG}}\cap \dsU_{\mathbf{\whL}}=\mathbf{\whT}e_{\mathbf{\whL}}  \subset \dsU_{\mathbf{\whG}}.
$
\end{enumerate}
\item The set $\{e_{\mathbf{\whL}}\ |\ \mathbf{\whL}\in\widehat{\cL}\}$ is the set of idempotent elements of the monoid $\mathds{U}_{\mathbf{\whG}}$.
\item The following map is a bijection: 
\begin{eqnarray*}
\widehat{\cL} & \lra& \dsU_{\mathbf{\whG}}/\mathbf{\whT}\\
\mathbf{\whL} & \lmapsto & \mathbf{\whT}e_{\mathbf{\whL}}= \dsS_{\mathbf{\whL}}.
\end{eqnarray*}

\end{enumerate}
\end{Prop}

\begin{proof}
These facts are contained in \cite[1.9 and \S 8]{CP21}, or can be deduced easily from \emph{loc. cit.}
\end{proof}

\subsection*{The simply connected case}

\begin{Pt}\label{affinecoord}
The structure of the standard scheme $\dsU_{\mathbf{\whG}}$ is particularly transparent when $\mathbf{\whG}$ is simply connected, i.e. when 
$$
X^*(\mathbf{\whT})^-=\bigoplus_{\widehat{\alpha}\in\widehat{\Delta}}\bbN w_0(\omega_{\widehat{\alpha}})
$$
where the $\omega_{\widehat{\alpha}}$ are the fundamental weights (and $w_0$ is the longest element of $W$). Then 
$(w_0(\omega_{\widehat{\alpha}}))_{\widehat{\alpha}\in\widehat{\Delta}}$
is a family of coordinates on $\dsU_{\mathbf{\whG}}$, defining an isomorphism 
$$
\xymatrix{
(w_0(\omega_{\widehat{\alpha}}))_{\widehat{\alpha}\in\widehat{\Delta}}:\dsU_{\mathbf{\whG}} \ar[r]^<<<<<{\sim} & (\bbA^{\widehat{\Delta}},\times)
}
$$
to the standard affine space on the set $\widehat{\Delta}$ (i.e. of dimension $|\widehat{\Delta}|$ after choosing a numbering of the simple roots) equipped its natural multiplicative monoid law. The torus $\mathbf{\whT}\subset \dsU_{\mathbf{\whG}}$ corresponds to  $\bbG_m^{\widehat{\Delta}}\subset \bbA^{\widehat{\Delta}}$, and writing $x_{\alpha}:=w_0(\omega_{\alpha})(x)$ for any geometric point $x\in \dsU_{\mathbf{\whG}}$, we have:
$$
x\in \dsU_{\mathbf{\whL}(\whI)}\quad\Longleftrightarrow\quad  \forall \widehat{\alpha}\in \widehat{\Delta}\setminus \widehat{I},\ x_{\widehat{\alpha}}\neq 0
$$
$$
x\in  \dsZ_{\mathbf{\whL}(\whI)\subset\mathbf{\whG}}\quad\Longleftrightarrow\quad \forall \widehat{\alpha}\in \whI,\ x_{\widehat{\alpha}}=0
$$
$$
x\in \dsS_{\mathbf{\whL}(\whI)}\quad\Longleftrightarrow\quad \forall \widehat{\alpha}\in \widehat{\Delta},\ x_{\widehat{\alpha}}
\left\{ \begin{array}{ll} 
\neq 0 & \textrm{if $\widehat{\alpha}\notin \whI$}\\
=0 & \textrm{if $\widehat{\alpha}\in \whI$}
\end{array} \right.
$$
$$
e_{\mathbf{\whL}(\whI)}=
\left\{ \begin{array}{ll} 
1 & \textrm{if $\widehat{\alpha}\notin \whI$}\\
0 & \textrm{if $\widehat{\alpha}\in \whI$}.
\end{array} \right.
$$
\end{Pt}

\subsection*{The case of a simply connected derived group}

In this section, we assume that the derived subgroup $\widehat{\bfG}^{\der}$ of $\widehat{\bfG}$ is simply connected. It allows to generalize the canonical system of coordinates on $\dsU_{\mathbf{\whG}}$ from the simply connected case \ref{affinecoord}, to a (no more canonical but) related system of coordinates. For example, it includes the example of $\widehat{\bfG}=GL_n$ (in which case there is even a canonical choice). Write $\widehat{\bfT}^{\der}:=
\widehat{\bfT}\cap \widehat{\bfG}^{\der}.$

\begin{Pt}
Let us consider the canonical image of $\widehat{\Delta}\subset X^{*}(\widehat{\bfT})$ in $X^{*}(\widehat{\bfT}^{\der})$ given by restriction to $\widehat{\bfT}^{\der}\subset \widehat{\bfT}$. This is a basis of the root system of $\widehat{\bfG}^{\der}$, and we will still denote it by $\widehat{\Delta}$. The corresponding set of simple coroots $\Delta\subset X_{*}(\widehat{\bfT}^{\scc})$ is a $\bbZ$-basis of the lattice $X_{*}(\widehat{\bfT}^{\der})$, and its dual basis in $X^{*}(\widehat{\bfT}^{\der})$ is the $\bbZ$-basis
$$
\{\overline{\omega_{\widehat{\alpha}}},\ \widehat{\alpha}\in\widehat{\Delta}\}\subset X^{*}(\widehat{\bfT}^{\der})
$$
of \emph{fundamental weights} associated to $\widehat{\bfG}$. In particular, $X^{*}(\widehat{\bfT})^0:=X^*(\widehat{\bfT})^-\cap X^*(\widehat{\bfT})^+$ is the kernel of the surjective group homomorphism
\begin{eqnarray*}
X^{*}(\widehat{\bfT}) & \lra & X^{*}(\widehat{\bfT}^{\der}) \\
\lambda & \lmapsto & \sum_{\widehat{\alpha}\in\widehat{\Delta}}\lan \lambda|_{\widehat{\bfT}^{\der}}, \alpha\ran\overline{\omega_{\widehat{\alpha}}},
\end{eqnarray*}
that is to say, $X^*(\widehat{\bfT})^0=X^{*}(\widehat{\bfG}^{\ab})$, where $\widehat{\bfG}^{\ab}$ is the torus $\widehat{\bfG}/\widehat{\bfG}^{\der}$. Moreover the fundamental weights are a basis of the monoid $X^{*}(\widehat{\bfT}^{\der})^+$, and the above map induces an isomorphism of monoids
\begin{eqnarray*}
X^{*}(\widehat{\bfT})^+/X^{*}(\widehat{\bfT})^0 & \lra & X^{*}(\widehat{\bfT}^{\der})^+ \\
\lambda & \lmapsto & \sum_{\widehat{\alpha}\in\widehat{\Delta}}\lan \lambda|_{\widehat{\bfT}^{\der}}, \alpha\ran\overline{\omega_{\widehat{\alpha}}},
\end{eqnarray*}
or equivalently, an isomorphism of monoids
\begin{eqnarray*}
X^{*}(\widehat{\bfT})^-/X^{*}(\widehat{\bfT})^0 & \lra & X^{*}(\widehat{\bfT}^{\der})^- \\
\lambda & \lmapsto & \sum_{\widehat{\alpha}\in\widehat{\Delta}}\lan \lambda|_{\widehat{\bfT}^{\der}}, w_0(\alpha)\ran w_0(\overline{\omega_{\widehat{\alpha}}}).
\end{eqnarray*}
In particular, in $X^{*}(\widehat{\bfT})^-$ we have:
$$
\Delta^{\perp}=(-\Delta)^{\perp}=(w_0(\Delta))^{\perp}=X^{*}(\widehat{\bfT})^0,
$$
i.e. 
$$
\mathds{S}_{\mathbf{\whG}}:=\Spec(\bbZ[\Delta^{\perp}])=\Spec(\bbZ[X^{*}(\widehat{\bfT})^0])=\widehat{\bfG}^{\ab}.
$$
\end{Pt}

\begin{Pt}\label{monoidbasis}
Let us choose a lift
$$
\{\omega_{\widehat{\alpha}},\ \widehat{\alpha}\in\widehat{\Delta}\}\subset X^{*}(\widehat{\bfT})^+.
$$
It gives a decomposition of the monoid $X^{*}(\widehat{\bfT})^-$ :
$$
X^{*}(\widehat{\bfT})^-=\bigoplus_{\widehat{\alpha}\in\widehat{\Delta}}\bbN w_0(\omega_{\widehat{\alpha}})\bigoplus X^{*}(\widehat{\bfT})^0,
$$
and hence an isomorphism of $\bbZ$-monoid schemes:
 $$
\xymatrix{
\dsU_{\mathbf{\whG}} \ar[r]^<<<<<{\sim} & \dsU_{\mathbf{\whG}^{\scc}}\times\mathds{S}_{\mathbf{\whG}}
\ar[r]^<<<<<{\sim} & \bbA^{\widehat{\Delta}}\times\widehat{\bfG}^{\ab}.
}
$$
\end{Pt}




\vskip10pt

\noindent {\small Tobias Schmidt, Bergische Universit\"at Wuppertal, Gaußstraße 20, 42119 Wuppertal, Germany, \newline {\it E-mail address: \url{toschmidt@uni-wuppertal.de}} 

\end{document}